\documentclass{scrartcl}

\usepackage[utf8]{inputenc}
\usepackage[T1]{fontenc}
\usepackage[english]{babel}
\usepackage{authblk}

\usepackage[dvipsnames]{xcolor}
\usepackage{graphicx,graphics}
\usepackage[font=sf, labelfont={sf,bf}, margin=0.1cm]{caption}
\usepackage{subfig}
\usepackage{tikz}
\usetikzlibrary{decorations.pathmorphing}

\usepackage{footnote}
\usepackage[a4paper,vmargin={3.5cm,3.5cm},hmargin={2cm,2cm}]{geometry}

\usepackage{xargs}
\usepackage{comment}
\usepackage{enumerate}
\usepackage{enumitem}
 
\usepackage[colorinlistoftodos,prependcaption]{todonotes}
\newcommandx{\rajouter}[2][1=]{\todo[inline,linecolor=red,backgroundcolor=red!25,bordercolor=red,#1]{#2}}
\newcommandx{\thiswillnotshow}[2][1=]{\todo[disable,#1]{#2}} 
\tikzset{notestyle/.append style={align=center}}
\usepackage[framemethod=TikZ]{mdframed}

\global\mdfdefinestyle{exampledefault}{%
linecolor=lightgray,linewidth=1pt,%
leftmargin=0.1cm,rightmargin=0.1cm,
}

\newcommand{\gr}[1]{{\color{gray} #1}}

\usepackage[normalem]{ulem}

\usepackage[backend=bibtex,style=alphabetic,giveninits=true,maxnames=99,maxalphanames=99,autopunct=false]{biblatex}
\usepackage{hyperref}
\hypersetup{colorlinks=true, citecolor=ForestGreen, breaklinks, urlcolor=black, linkcolor=Black}
\usepackage{amsmath}
\usepackage{cleveref}
\AtEveryCite{\color{ForestGreen}}

\usepackage{amsfonts,amssymb,amsthm,mathrsfs, cancel, amsmath}
\usepackage{wasysym}
\usepackage{esint} 
\usepackage{upgreek}
\usepackage{dsfont}
\usepackage{braket}
\usepackage{mathtools}

\newcommand\brak[1]{\big \langle #1 \big\rangle}

\newcommand{\N}{\mathbb{N}}

\newcommand{\C}{\mathbb{C}}

\newcommand{\Z}{\mathbb{Z}}
\newcommand{\Q}{\mathbb{Q}}
\newcommand{\QQ}{\overline{\mathbb{Q}}}
\newcommand{\K}{\mathbb{K}}

\renewcommand{\P}{\mathbb{P}}
\newcommand{\dinf}{\mathrm{d}}

\renewcommand{\epsilon}{\varepsilon}
\newcommand\wt[1]{\widetilde{#1}}
\newcommand\wh[1]{\widehat{#1}}
\newcommand{\notdiv}{\hspace{-0.25em}\not\hspace{0.25em}\mid}
\newcommand\ul[1]{\underline{#1}}
\newcommand{\w}{\textnormal{w}}

\makeatletter
\newif\iflibus@sansmath
\makeatother
\DeclareFontEncoding{LS1}{}{}
\DeclareFontSubstitution{LS1}{libertinust1math}{m}{n}
\DeclareSymbolFont{LettersLibertinus}{LS1}{libertinust1math}{m}{it}
\DeclareMathSymbol{\tau}{\mathord}{LettersLibertinus}{"1C}

\DeclareFontFamily{OT1}{pzc}{}
\DeclareFontShape{OT1}{pzc}{m}{it}{<-> s * [1.10] pzcmi7t}{}
\DeclareMathAlphabet{\mathpzc}{OT1}{pzc}{m}{it}

\newtheorem{thm}{Theorem}[section]
\newtheorem{prop}[thm]{Proposition}
\newtheorem{lem}[thm]{Lemma}
\newtheorem{cor}[thm]{Corollary}
\newtheorem{defn}{Definition}[section]
\newtheorem{rem}{Remark}[section]

\begin{document}

\title{Enumeration of general planar hypermaps with an alternating boundary}
\author[1,2]{Valentin Baillard}
\author[3]{Ariane Carrance}
\author[1,4]{Bertrand Eynard}

\affil[1]{\normalsize Université Paris-Saclay, CNRS, CEA, Institut de physique théorique, 91191, Gif-sur-Yvette, France.}
\affil[2]{Institut de Mathématiques d'Orsay, Université Paris-Saclay, 91400, Orsay, France.}
\affil[3]{Fakultät für Mathematik, Universität Wien, 1090 Wien, Österreich.}
\affil[4]{CRM, Centre de recherches math\'ematiques  de Montr\'eal, Montr\'eal, QC, Canada.}

\date{\today}

\maketitle

\begin{abstract}
In this paper, we extend the enumerative study of planar hypermaps with an \emph{alternating} boundary introduced in~\cite{bouttier-carrance}. In that article, an explicit rational parametrization was obtained for the associated generating function in the case of $m$-constellations, using a variant of the kernel method. We develop here a new strategy to obtain an algebraic equation in the general case, which includes maps decorated by the Ising model, through a classical many-to-one correspondence. One of the main steps of our strategy is the simultaneous elimination of two {catalytic} variables. We then apply this strategy to the case of Ising quadrangulations, where we obtain an explicit rational parametrization. As a consequence, we show that some notable properties of the constellations case are no longer satisfied in general.
\end{abstract}

\tableofcontents

\section{Introduction}

\subsection{Context}

The enumerative theory of planar maps has been an active topic in combinatorics since its inception by Tutte in the sixties, through a series of seminal papers, including~\cite{tutte-trig,tutte-slicings,tutte-gene, tutte68}. An account of some of its more recent developments may be found in the review by Schaeffer~\cite{schaeffer-handbook}. Many such developments were motivated by connections with other fields: theoretical physics, algebraic and enumerative geometry, computer science, probability theory... 

In the physics literature, the problem of enumerating planar maps emerged in connection with matrix models: this was initiated by t’Hooft~\cite{thooft} and further developed in particular by Brezin, Itzykson, Parisi and Zuber~\cite{bipz}. A notable motivation comes from two-dimensional quantum gravity, see~\cite{dfgzj} for a detailed review. In the language of physics, usual maps provide either a model of pure gravity without  matter, or with some very restricted type of matter. It is therefore natural to consider more generally maps endowed with a model of statistical mechanics. In particular, we consider here hypermaps, which are maps whose faces are bicolored (say, in black and white) in such a way that no two faces of the same color are adjacent. They are related through a classical many-to-one correspondence to maps endowed with an Ising model on their faces (see Section~\ref{sec:corresp-ising}). The enumeration of hypermaps was thus first addressed in the physics literature through the two-matrix model~\cite{iz,mehta,douglas,staudacher,dkk}, leading to the identification of the critical exponents of the Ising model on random maps in~\cite{kazakov,boulatov-kazakov}. While it is fully rigorous to equate the correlation functions of \emph{formal} matrix models with generating functions of maps (see for instance~\cite{eynard-formal-matrix}), the connection with \emph{convergent} matrix models is much more subtle (see~\cite{guionnet-maurel-segala}).

To enumerate maps, the standard approach consists in studying the effect of the removal of an edge. Nowadays, this operation is often called ``peeling'', see for instance~\cite{curien-stflour}. In this paper, we also use the term splitting, that comes from matrix models, see Section~\ref{sec:tutte as splitting}. In order to turn the peeling into equations, one needs to keep track of one or more auxiliary parameters, called \emph{catalytic variables}~\cite{BoJe06}, and typically corresponding to boundary lengths. In the context of hypermaps, a further complication occurs, since one needs to keep track of boundary conditions: these encode the colors of the faces incident to the boundary. The most tractable boundary condition is the so-called Dobrushin boundary condition, which consists in having the boundary made of two parts: one part is incident to white faces, and the other to black faces. The Dobrushin boundary condition has the key property that it is invariant under peeling, provided that we always peel an edge at the white-black interface. The resulting equations are solved in~\cite{Eyn_2003,eynard} and in~\cite{chen-turunen}. Knowing how to treat Dobrushin boundary conditions, one may then consider ``mixed'' boundary conditions~\cite{eyn_2005} where there is a prescribed number of white-black interfaces. However, this approach seems to become intractable when the number of interfaces gets large. An earlier work by Bouttier and the second author~\cite{bouttier-carrance} introduced the \emph{alternating} boundary condition, which corresponds to the extreme situation where the number of interfaces is maximal, that is to say when white and black faces alternate along the boundary. 

Note that~\cite{eyn_2005,eyn_2008,eynard} also treats the case of mixed boundary conditions in higher genera and with more than one boundary, through the formalism of \emph{topological recursion}. While the present paper only deals with the disk topology, we will touch upon the notion of \emph{spectral curve}, as constructed in topological recursion, when discussing properties our generating functions, see Section~\ref{subsec:constell}. It is natural to conjecture that, like many other map enumeration problems, hypermaps with an alternating boundary also fall under the general framework of topological recursion. We discuss this in more detail in Sections~\ref{subsec:constell} and~\ref{sec:ccl}.

Beyond the interesting combinatorial phenomena that we uncover in the present paper, there are some external motivations to study the alternating boundary condition. First, the initial motivation of~\cite{bouttier-carrance} came from the probabilistic study of bicolored triangulations. Indeed, enumeration asymptotics of the alternating boundary condition obtained in~\cite{bouttier-carrance} were a crucial ingredient in a related work by the second author~\cite{carrance-trig-eul}, that proved that large random uniform bicolored triangulations converge to the \emph{Brownian sphere}. In the case of the Ising model, the alternating boundary condition is also natural to consider when studying the \emph{antiferromagnetic} regime. Indeed, for bipartite lattices, one can transform an antiferromagnetic model into a ferromagnetic one, by flipping the spins of odd sites (see for instance~\cite[Section 3.10.5]{friedli-velenik-book}), and this transforms the alternating boundary condition into the classical monochromatic one. Moreover, the Ising model on random maps is related to \emph{Calogero-Moser spin chains} by a bulk/boundary correspondence, see for instance~\cite[Section 6.1]{ekr}. Through this correspondence, the partition function of the alternating condition is related to \emph{disorder operators} of the spin chain. Thus, combinatorial results on the alternating condition open up new directions in the study of these statistical mechanics models.

Before presenting our main results, let us mention that there exist many bijective approaches to the enumeration of hypermaps. While most of them focus on the monochromatic boundary~\cite{bm-s-ising, bdg, bdg-hard-particles, bouttier-fusy-guitter, bernardi-fusy, albenque-menard-tokka}, the Dobrushin boundary condition was recently recovered using the bijective method of the slice decomposition~\cite{albenque-bouttier-slices}. We will address the adaptation of the slice decomposition to the alternating boundary in a future work~\cite{slices-alt}.

\subsection{Main results}

Recall that a \emph{planar map} is a connected multigraph drawn on the sphere without edge crossings, and considered up to continuous deformation. It consists of vertices, edges, faces, and corners. A \emph{boundary} is a distinguished face, which we often choose as the infinite face when drawing the map on the plane. It is assumed to be \emph{rooted}, that is to say there is a distinguished corner along the boundary, depicted as gray arrow in our illustrations. The other faces are called \emph{inner faces}. A \emph{hypermap with a boundary} is a planar map with a boundary, where every inner face is colored either black or white, in a such a way that adjacent inner faces have different colors. Note that the boundary does not have a specified color: it may be incident to faces of both colors. Let $\w$ be a word on the alphabet $\{\bullet, \circ\}$, and $\ell$ be the length of $\w$. A hypermap with boundary condition $\w$ is a hypermap with a boundary of length $\ell$ such that, when walking in clockwise direction around the map (starting at the root corner),
\begin{enumerate}
\item if the $i$-th visited edge $(i = 1, \dots , \ell)$ is incident to an inner face, then this face is white if $\w_i = \bullet$ and black if $\w_i = \circ$,
\item  if an edge is incident to the boundary on both sides (i.e. it is a \emph{boundary bridge}), and is then visited twice, say at steps i and j, then $\w_i \neq \w_j$. 
\end{enumerate}
In other words, $\w_i$ codes the color of the boundary face as seen from the $i$-th edge.

\begin{rem}
Looking at a given hypermap with a boundary, there is a possible ambiguity for the associated boundary condition if there are bridges (see Figure~\ref{fig:gene-and-mono-ex}), but this causes no issue for our generating functions, as they count hypermaps that \emph{admit} given boundary conditions.
\end{rem}

\begin{figure}[htp]
\centering
\includegraphics[height=4cm]{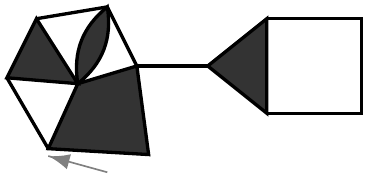} \hfill \includegraphics[height=4cm]{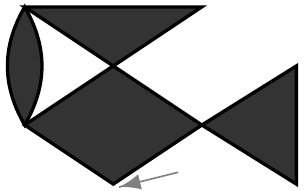}
\caption{Left, a hypermap that admits the boundary conditions $\circ\bullet\circ\bullet\bullet\bullet\circ\bullet\bullet\bullet\circ\circ\circ$ and $\circ\bullet\circ\bullet\bullet\circ\circ\bullet\bullet\bullet\circ\bullet\circ$. Right, a hypermap with a monochromatic boundary.}
\label{fig:gene-and-mono-ex} 
\end{figure}

A \emph{hypermap with a white (resp. black) monochromatic boundary} is a hypermap with boundary condition $\circ \circ \dots \circ$ (resp. $\bullet\bullet \dots \bullet$). These are just hypermaps in the sense of~\cite{bernardi-fusy}. A \emph{hypermap with an alternating boundary} is a hypermap with boundary condition $\circ\bullet\circ\bullet\dots\circ\bullet$. Note that a hypermap with a monochromatic boundary cannot have a boundary bridge, but may have pinch points (see Figure~\ref{fig:gene-and-mono-ex}). For a hypermap with an alternating boundary, bridges are possible, but they must necessarily have black neighboring boundary edges on one side, and white ones on the other (see Figure~\ref{fig:alt-non-alt-ex}).

Throughout this paper, we consider generating functions for hypermaps whose white (resp. black) face degrees are upper bounded by some integer $d$ (resp. $\wt{d}$). We will keep track of the number of faces of each color and degree with variables $t_2, \dots, t_d,\wt{t}_2,\dots,\wt{t}_{\wt d}$. The following polynomials in $\Q[t_2,\dots,t_d][x]$ (resp. $\Q[\wt{t}_2,\dots,\wt{t}_{\wt{d}}][y]$):
\begin{equation}
V(x) = - \sum_{2\leq i \leq d} \frac{t_i}{i}x^i, \qquad \wt{V}(y) = - \sum_{2\leq i \leq \wt{d}} \frac{\wt{t}_i}{i}y^i
\end{equation}
will be useful quantities, that we call \emph{potentials} by analogy with matrix models.

\begin{figure}[htp]
\centering
\includegraphics[height=4cm]{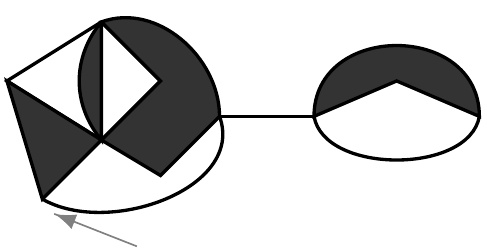} \hfill \includegraphics[height=4cm]{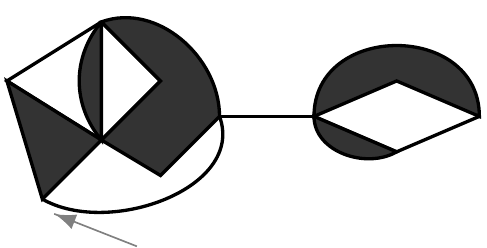}
\caption{Left, a hypermap with an alternating boundary. Right, a hypermap whose boundary is not alternating.}
\label{fig:alt-non-alt-ex} 
\end{figure}
 
Let $F_r$ denote the generating function of hypermaps with an alternating boundary of length $2r$, counted with a weight $t$ per vertex, weight $c^{-1}$ per edge, $t_i$ per white inner face of degree $i$, and $\wt{t}_j$ per black inner face of degree $j$. By convention, $F_0=1$ (corresponding to the trivial vertex map). Our main series of interest is
\begin{equation}
\wh{f}(\omega)=\wh{f}(t,c,t_2,\dots,t_d,\wt{t}_2,\dots,\wt{t}_{\wt{d}}; \omega):=\sum_{r\geq 0}\frac{F_r}{\omega^{r+1}}-c.
\end{equation}
The extra $-c$ might seem arbitrary, but it will help simplify some formulae in what follows. We define similarly $W_p$ as the generating function of hypermaps with a white monochromatic boundary of length $p$, with the same weights, and
\begin{equation}
W(x)=W(t,c,t_2,\dots,t_d,\wt{t}_2,\dots,\wt{t}_{\wt{d}};x):=\sum_{r\geq 0}\frac{W_p}{\omega^{p+1}}, \quad Y(x):=\frac{1}{c}(W(x)-V'(x))
\end{equation}

We are now ready to state our main results.

\begin{thm}
For any $d, \wt d\geq2$, $\wh{f}$ is algebraic over $\QQ((t))(c,t_1,\dots,t_d,\wt{t}_1,\dots,\wt{t}_{\wt d},\omega)$. Moreover, we have an explicit strategy to obtain its annihilating polynomial, that does not rely on the kernel method or its generalizations, see Theorem~\ref{thm:algebraicity} for a more precise statement.
\end{thm}

Note that it was mentioned in~\cite{bouttier-carrance} that the general case is also amenable to the kernel method. We give details on this in Appendix~\ref{sec:appendix-alg-by-bmj}, and compare it with our strategy in Section~\ref{sec: algebraicity-2-strats}. In particular, in the case of Ising quadrangulations, our strategy is much simpler than the kernel method, and leads to the following:

\begin{thm} For the specialization $t_2=\wt{t}_2, t_4=\wt{t}_4, t_k=\wt{t}_k=0$ if $k\neq 2, 4$, i.e. symmetric Ising quadrangulations, we have the following explicit rational parametrization of $(\omega,\wh{f}(\omega))$ in terms of a formal variable $h$:
\begin{equation}
\begin{aligned}
\omega_\textnormal{sym}(h) &= \frac{(\alpha_3 h +\gamma)^2\gamma^3h}{(\gamma h + \alpha_3)^2\alpha_3}\\
\widehat{f}_\textnormal{sym}(h) &=-c\frac{(\alpha_3 \gamma h^2 + h(\alpha_1^2 - 2\alpha_3\gamma) + \alpha_3 \gamma) (\gamma h +\alpha_3)^3}{(\alpha_3 h + \gamma) \gamma^4 (h - 1)^2h^2},
\end{aligned}
\label{eq:parametrization-for-Ising-quadrangulations-symmetric-f-omega}
\end{equation}
where $\gamma, \alpha_1, \alpha_3$ are algebraic functions in $c,t_2,t_4,t$ that are determined by 
\[
t_4  = \frac{c \alpha_3}{\gamma^3} , \quad t_2  = c\alpha_1\frac{\gamma - 3\alpha_3}{\gamma^2}, \quad t= c(\gamma^2-\alpha_1^2-3\alpha_3^2).
\]
\label{thm:ratio-param-Ising-quad-sym}
\end{thm}

One remarkable feature of the case of $m$-constellations that was solved in~\cite{bouttier-carrance}, is that $(\omega,\wh{f})$ can be parametrized as rational functions of $(x,Y(x))$, that satisfy the \emph{kernel relation} $\omega\wh{f}(\omega)+cxY(x)=0$ (see the detailed discussion in Section~\ref{subsec:constell}). As a consequence of Theorem~\ref{thm:ratio-param-Ising-quad-sym}, we show that this no longer holds in the general case:

\begin{cor}
\label{cor:non-co-rat}
In general, the curves $(\omega,\widehat{f}(\omega))$ and $(x,Y(x))$ do not admit a joint rational parametrization that satisfies the kernel relation.
\end{cor}

The rest of the paper is organized as follows. In Section~\ref{sec:reminders}, we recall some important results from previous works. In Section~\ref{sec:reminder-mono}, we recall properties of the monochromatic boundary from~\cite{eynard} and introduce the bivariate polynomial $E(x,y)$ called the spectral curve, and in Section~\ref{subsec:constell}, we recall some relevant results from~\cite{bouttier-carrance} and discuss some consequences that were not stated there. In Section~\ref{sec:def-split}, we set up the notation for all the auxiliary functions that we will need, and recall the correspondence between hypermaps and the Ising model on maps. We also establish a Tutte decomposition/splitting procedure for general boundary conditions in Proposition~\ref{prop:splitting-procedure}. In Section~\ref{sec:general-master-eq}, we apply Proposition~\ref{prop:splitting-procedure} to various auxiliary functions, leading to an equality between the spectral curve $E(x,y)$ that does not depend on $\omega$, and another bivariate polynomial $Q(x,y)$ whose coefficients are explicit functions of $\omega$, $\wh{f}(\omega)$ and some other auxiliary generating functions, as stated in Proposition~\ref{prop:Q=E}. Since both polynomials have the same degree, same leading coefficient and vanish on $(x,Y(x))$, they must be equal, and requiring the vanishing of coefficients of $Q-E$ thus gives a polynomial system on $\wh{f}(\omega)$ and related functions. We then prove that this system has non-zero Jacobian determinant evaluated at those functions, which gives Theorem~\ref{thm:algebraicity}. We also compare this strategy to the kernel method. We then show in Proposition~\ref{prop:tutte-bc-from-splitting} that we recover the Tutte equation of~\cite{bouttier-carrance} from our splitting equations. In Section~\ref{sec:ising-quad}, we apply the strategy of Theorem~\ref{thm:algebraicity}: we obtain an explicit master equation in the general (non-symmetric) case in Section~\ref{sec:master-eq-ising-quads}, and solve it for the symmetric case in Section~\ref{sec:rat-param-ising-quads}, yielding Theorem~\ref{thm:ratio-param-Ising-quad-sym}. We conclude with a list of open questions in Section~\ref{sec:ccl}. In Appendix~\ref{sec:appendix-alg-by-bmj}, we give details on the application of the kernel method to the Tutte equation of~\cite{bouttier-carrance} for general potentials.

\section{Reminders on related works}
\label{sec:reminders}

\subsection{Hypermaps with a monochromatic boundary}
\label{sec:reminder-mono}

Let us recall some important results from~\cite{Eyn_2003} summarized in~\cite[Chapter 8]{eynard}. Note that we are almost using the exact same convention, save for the fact that we consider proper hypermaps, while~\cite[Chapter 8]{eynard} describes the Ising model via bicolored digons in maps with monochromatic edges. The only difference that this entails is that it introduces in our formualae a minus sign in front of the parameter $c$ that counts edges.

Hence we define $Y(x)=\frac{1}{c}(W(x)-V'(x))$, where $W(x)=W^{(0)}_1(x)$ in~\cite{eynard}. It is an algebraic function of $x, t, \{t_i\}, \{\widetilde{t}_j\}$, and $(x,Y(x))$ admits the following explicit rational parametrization:

\begin{thm}[Theorem \cite{Eyn_2003} and 8.3.1 in~\cite{eynard}]
For a formal variable $z$, we have
\begin{equation}
  x(z)=\gamma z + \sum_{k=0}^{\widetilde{d}} \alpha_k z^{-k}, \qquad  Y(x(z))=y(z)=\gamma z^{-1} + \sum_{k=0}^{d} \beta_k z^{k},
\label{eq:xy-param}
\end{equation}
where the coefficients $\alpha_k,\beta_k,\gamma$ are algebraic functions of $ t, \{t_i\}, \{\widetilde{t}_j\}$, that are characterized by:
\begin{equation}
    V'(x)+cy = \frac{t}{\gamma z} + O(z^{-2}), \ \ \tilde V'(y)+cx = \frac{tz}{\gamma} + O(z^{2}), \ \ \gamma^2=-\frac{ct}{t_2\widetilde{t}_2-c^2}+O(t^2).
\label{eq:system-on-alpha-beta-gamma}
\end{equation}
\end{thm}

Equivalently,  $Y(x)$ is characterized as the unique power series solution of $E(x,y)=0$, where the polynomial $E(x,y)$ is equal to (see~\cite[Theorem 8.3.2]{eynard}):
\begin{equation}
 E(x,y) = -\frac{(-1)^{d}c^2}{\gamma^{d+\tilde d-2}}\det
 \begin{pmatrix}
  \gamma & \alpha_0-x &  \alpha_{1} & \ldots  & \alpha_{\widetilde{d}-1} &  & & & \cr
 & \gamma & \alpha_0-x & \alpha_{1} &\ldots  & \alpha_{\widetilde{d}-1} &  & &  \cr
 & & \ddots &  &  &  & \ddots &  &   \cr
 & & & \ddots &  &  &  & \ddots &     \cr
 & & & & \gamma & \alpha_0-x &  \alpha_{1} & \ldots &\alpha_{\widetilde{d}-1}     \cr
  \beta_{d-1} &  \ldots & \beta_1 & \beta_0-y & \gamma & & & \cr
 & \beta_{d-1}& \ldots & \beta_1 & \beta_0-y & \gamma & & \cr
 & & \ddots &  &  &  &  & \ddots & \cr
 & & & \beta_{d-1} & \ldots & \beta_1 & \beta_0-y & \gamma & \cr
 \end{pmatrix}.
\label{eq:master-equation-monochromatic-case}
\end{equation}
It can also be written as
\begin{equation}
  E(x,y)=(V'(x)+cy)(\tilde V'(y)+cx)+ \frac{1}{c}P(x,y)- ct,
\label{eq:E-def}
\end{equation}
where
\begin{equation}
  P(x,y)= \brak{\frac{V'(x)-V'(A)}{x-A}\frac{\tilde V'(y)-\tilde V'(B)}{y-B}}.
\end{equation}

\subsection{Previous results on the alternating boundary condition}
\label{subsec:constell}

Let us recall and revisit some results from~\cite{bouttier-carrance}. This work established a Tutte equation (equation (27) there) on the generating function that was called $M(x,w)$ there, and that we denote here by $\underline{M}(x,w)$ to avoid ambiguity. As the notational conventions of~\cite{bouttier-carrance} are slightly different from those we use here (see Section~\ref{sec:general-master-eq} for more details), this generating function is related to the ones of the present paper by
\[
\ul M(x,w)=\frac{1}{w}M\left(x,\frac{1}{w}\right)-W(x),
\]
and the generating function of interest in~\cite{bouttier-carrance} is $A(w)=tw+ \ul M(0,w)=f_{00}(1/w)$. 

Then, equation (27) in~\cite{bouttier-carrance} is written as
\begin{equation}
K(w,x)\ul M(w,x) = \ul R(w,x),
\label{eq:bouttier-carrance-eq-notation}
\end{equation}
with 
\begin{align*}
K(x,w) &= 1 - A(w) - w x Y(x),\\
\ul R(x,w)&= w x W(x)Y(x) - w\sum_i t_i \left[x^i(\ul M(x,w)+W(x))\right]_{x^{\geq 0}}.
\end{align*}
We compare this equation with the functional equations that we derive in the present paper in Section~\ref{sec:bouttier-carrance-eq-rederived}. Let us focus for the moment on~\eqref{eq:bouttier-carrance-eq-notation} and its role in~\cite{bouttier-carrance}. In full generality, the remainder term $\ul R$ is hard to tackle. However, it was mentioned in~\cite[Section 3]{bouttier-carrance} that~\eqref{eq:bouttier-carrance-eq-notation} can be rewritten in a form where one can apply~\cite[Theorem 3]{BoJe06}, and conclude that $\ul M(x,w)$, hence $A(w)$, is algebraic in $w, x, W(x), t, t_i, \wt{t}_j$. As this argument was not detailed in~\cite{bouttier-carrance}, and we develop here another path to the algebraicity of $A(w)$, we detail in Appendix~\ref{sec:appendix-alg-by-bmj} how~\eqref{eq:bouttier-carrance-eq-notation} can be rewritten in such a form, so that we can compare the two strategies in detail in Section~\ref{sec: algebraicity-2-strats}.

In the special case of \emph{$m$-constellations} (i.e. for the specialization $t_k=0$ for $k \neq m$ and $\wt{t}_\ell=0$ if $m \notdiv \ell$), $\ul R(x,w)$ simplifies greatly (see~\cite[Section 4]{bouttier-carrance}):
\[
\ul R(x,w)=wxY(x)^2-wx^mY(x)-x^{m-1}A(w).
\]
Going back to equation~\eqref{eq:bouttier-carrance-eq-notation}, one can then perform a variant of the \emph{kernel method} (see for instance~\cite{BoJe06} for a detailed exposition, and Appendix~\ref{sec:appendix-alg-by-bmj} for a more classical application of it to solve~\eqref{eq:bouttier-carrance-eq-notation} in the general case). Indeed, instead of eliminating the \emph{catalytic variable} $x$ from~\eqref{eq:bouttier-carrance-eq-notation}, one can consider the unique formal power series $\mathpzc{w}(\xi)$ in $\xi^{-1}$ such that
\begin{equation}
\mathpzc{w}(\xi)=\frac{1-A(\mathpzc{w}(\xi))}{\xi}.
\label{eq:kernel-relation}
\end{equation}
Since the substitution $\ul M(x,\mathpzc{w}(xY(x)))$ is well-defined, by~\eqref{eq:bouttier-carrance-eq-notation} we deduce that
\[
\ul R(x,\mathpzc{w}(xY(x)))=0
\]
as well. This gives a linear system of two equations on the series $\mathpzc{w}(xY(x))$ and $A(\mathpzc{w}(xY(x)))$, which determines them as

\begin{thm}[Explicit parametrization for $m$-constellations, Theorem 1 in~\cite{bouttier-carrance}]
\[
\mathpzc{w}(xY(x))=\frac{x^{m}}{(xY(x))^2}, \qquad A(\mathpzc{w}(xY(x)))=1-\frac{x^m}{xY(x)}.
\]
\end{thm}

One might worry that these formulae do not depend only on $xY(x)$, but also on $x^m$. However, $xY(x)$ is actually a Laurent series in $x^{-m}$: $xY(x)=\Xi(x^m)=x^m +t+...$, and one can perform a Lagrange inversion, to express any analytic function in $x^{-m}$, as an analytic function in $(xY(x))^{-1}$. Thus, the above formulae properly define the functions $\mathpzc{w}(\Xi), A(\Xi)$.

One can perform the additional substitution $x=x(z)$, or more precisely $\Xi(x^m)=\xi(z^m)$, with $\xi(z^m)=x(z)y(z)$ being a uniquely defined Laurent polynomial in $z^m$. Since $\chi(z^m)=x(z)^m$ is also well-defined as a Laurent polynomial in $z^m$, this yields the rational parametrization
\begin{equation}
\mathpzc{w}(s)=\frac{\chi(s)}{\xi(s)^2}, \qquad a(s)=1-\frac{\chi(s)}{\xi(s)},
\label{eq:const-param-ratio-s}
\end{equation}
which was stated in Theorem 1 in~\cite{bouttier-carrance}.

Thus, in the case of $m$-constellations, the algebraic curve $(w,A(w))$ can be expressed as a pair of rational functions over the algebraic curve $(x(z),y(z))$ corresponding to the monochromactic boundary condition. Moreover, this parametrization of $(w,A(w))$ over $(x(z),y(z))$ satisfies the kernel relation~\eqref{eq:kernel-relation}.

Let us now discuss another remarkable feature of this parametrization, that was not stated in~\cite{bouttier-carrance}. Let us switch the variable $w$ to the inverse variable $\omega$ (as is the case in the present paper), and write $\Omega(s)$ for $\mathpzc{w}(s)^{-1}$. Geometrically, the mapping $\varphi: \C\P^1 \to \C\P^1$, $\varphi(z)=z^m$ is a branched covering of degree $m$. To the curve $(x(z),y(z))$ (and its explicit parametrization), one associates a canonical symplectic form $\sigma=\dinf x\wedge \dinf y$  -- note that this has deeper implications as this curve is upgraded to a \emph{spectral curve} through the formalism of topological recursion, when one generalizes the enumeration problem to more general topologies, see~\cite{eynard}. Likewise, we define $\tau=\dinf\Omega\wedge\dinf a$ from $(\Omega(s),a(s))$ --  here the enumeration problem for other topologies has yet to be solved, and thus one cannot really talk of a spectral curve for now, see the discussion in Section~\ref{sec:ccl}. Then, we have the following property:
\begin{prop}
We have
\[
\varphi^*\tau=m\cdot \dinf x \wedge \dinf y,
\]
i.e. the pullback of $\tau$ by $\varphi$ is exactly equal to $\sigma$ multiplied by the degree of the covering $\varphi$.
\label{prop:sympl-nice}
\end{prop}

\begin{proof}
We have
\[
\begin{aligned}
\varphi^*\tau&=\dinf(\Omega\circ\varphi)\wedge\dinf(a\circ\varphi)=\dinf\left(\frac{y^2}{x^{m-2}} \right)\wedge\dinf\left(1-\frac{x^{m-1}}{y} \right)\\
&=\frac{1}{x^{m-1}}\left(2xy\dinf y -(m-2)y^2\dinf x \right)\wedge \frac{1}{y^2}\left( -(m-1)yx^{m-2}\dinf x - x^{m-1}\dinf y\right)\\
&=m\cdot \dinf x \wedge \dinf y.
\end{aligned}
\]
\end{proof}

Given this very nice geometric property, it is natural to wonder whether, when considering higher genera, the generating functions of the alternating boundary condition can also be expressed nicely in terms of the ones of the monochromatic condition. If this is the case, this would be analogous to the \emph{symplectic invariance} phenomenon, that occurs for spectral curves related by symplectomorphisms (see for instance~\cite{fully-simple-ciliated-maps}).

These properties of the $m$-constellation case were a strong motivation to study the alternating boundary condition on more general hypermaps. In this paper, we give some partial answers for the general case for the disk topology. In particular, we show that on more general hypermaps, $(\omega,A(\omega^{-1}))$ cannot be written as rational functions of $(x,y)$ satisfying the kernel relation~\eqref{eq:kernel-relation}, see Corollary~\ref{cor:non-co-rat}.

\section{Preliminaries}
\label{sec:def-split}

\subsection{General definitions}
\label{subsec:def-GFs}

As we will handle auxiliary generating functions with various types of boundary conditions, it will be convenient to borrow some notational conventions from formal matrix models.

\begin{defn}
Let $d\in \N$, $\widetilde{d}\in \N$ and let $P$ be a word over $\{A,B\}$: we identify such a word with a non-commutative monomial in $A,B$, and as a boundary condition for hypermaps through the correspondence $A \mapsto \bullet, B \mapsto \circ$. We denote by $\textnormal{Bic}^{d,\widetilde{d}}(P)$ the set of hypermaps of the disk with boundary condition $P$, and whose white (resp. black) faces are of degree at most $d$ (resp. $\widetilde{d}$).

We then define
\[
\brak{P} :=\sum_{\mathfrak{m} \in\textnormal{Bic}^{d,\widetilde{d}}(P)} t^{v(\mathfrak{m})}c^{-e(\mathfrak{m})}\prod_{2\leq i\leq d}t_i^{f_\circ^i(\mathfrak{m})}\prod_{2\leq j\leq \widetilde{d}}\widetilde{t}_j^{f_\bullet^j(\mathfrak{m})},
\]
where $v(\mathfrak{m})$ is the number of vertices of $\mathfrak{m}$, $e(\mathfrak{m})$ its number of edges, and $f_\square^i(\mathfrak{m})$ its number of faces of color $\square$ and degree $i$. In other words, $\brak{P}$ is the generating function of maps in $\textnormal{Bic}^{d,\widetilde{d}}(P)$, enumerated with respect to their number of vertices, edges, and faces of given degree.
\\
We extend this definition by linearity first to any non-commutative polynomial in $A,B$, then to any formal series in the non-commutative variables $A$ and $B$. We take the convention that $\textnormal{Bic}^{d,\widetilde{d}}(1)=\{ \cdot\}$ (the trivial vertex map) and thus $\brak{1}=t$, and that $\brak{\varnothing}=1$ (and $\textnormal{Bic}^{d,\widetilde{d}}(\varnothing)=\varnothing$).
\label{def:traces-as-map-amplitudes}
\end{defn}

\begin{rem}
    Note that linearity allows to write generating series such as  $$ \frac{1}{x-A} = \sum_{k=0}^\infty x^{-k-1} A^k $$

Thus, our main generating function $\wh{f}$ is equal to:
\[
\wh{f}=\frac{1}{\omega}\sum_{k\geq 0}\frac{1}{\omega^k}\brak{(AB)^k}-c=\brak{\frac{1}{\omega -BA}} -c.
\]
\end{rem}

We define similarly the following auxiliary generating functions:
\begin{align}
f_{ij}(\omega):=\brak{A^i\frac{1}{\omega - BA}B^j}&, \qquad D_{ij}:=\brak{A^iB^j} \text{ for } i,j \in \Z_{\geq 0} \nonumber\\
W(x):=\brak{\frac{1}{x-A}}&, \qquad M(x,\omega):=\brak{\frac{1}{x-A}\frac{1}{\omega - BA}}, \nonumber\\
R(x,\omega):=\brak{\frac{V'(x) - V'(A)}{x-A}\frac{1}{\omega - BA}}&, \qquad S_+(x,y,\omega):=\brak{\frac{V'(x) - V'(A)}{x-A}\frac{1}{\omega - BA}\frac{1}{y-B}}, \nonumber\\
\widehat{P}(x,y,\omega)&:=\brak{\frac{V'(x) - V'(A)}{x-A}\frac{1}{\omega - BA}\frac{\widetilde{V}'(y)-\widetilde{V}'(B)}{y-B}},
\label{eq: auxiliary fcts def}
\end{align}
and, for functions that are not symmetric in $A \leftrightarrow B$ and $x \leftrightarrow y$, we denote the mirrored function with a tilde ($\widetilde{W}(y)$, etc.). We will also make use the generating functions of the monochromatic boundary condition that were defined in Section~\ref{sec:reminder-mono}: $cY(x)=W(x)-V'(x), \, cX(y)=\widetilde{W}(y)-\widetilde{V}'(y)$.

\subsection{Correspondence with the Ising model}
\label{sec:corresp-ising}

Let us now turn to the connection between hypermaps and the Ising model on maps, i.e. maps carrying +/- spins on their faces. A word $P$ on $\{A,B\}$ can also be seen as a word over $\{+,-\}$, and thus as a boundary condition for Ising maps. Let $\textnormal{I}_{d,\wt d}(P)$ be set of Ising planar maps with boundary condition $P$ and $+$ (resp. $-$) faces of degree at most $d$ (resp. $\wt d$). Let us define the generating function of $\textnormal{I}_{d,\wt d}(P)$ as
\[
I(P)=\sum_{\mathfrak{m}\in\textnormal{I}_{d,\wt d}(P)}t^{v(\mathfrak{m})}\prod_{3\leq i\leq d}t_i^{f_{+}^i(\mathfrak{m})}\prod_{3\leq j\leq \widetilde{d}}\widetilde{t}_j^{f_{-}^j(\mathfrak{m})}c_{++}^{e_{++}(\mathfrak{m})}c_{--}^{e_{--}(\mathfrak{m})}c_{+-}^{e_{+-}(\mathfrak{m})},
\]
where $e_{\square\triangle}(\mathfrak{m})$ is the number of edges of type $\square\triangle$ in $\mathfrak{m}$. Then we have

\begin{prop}
With the matching of parameters
\[ 
\begin{pmatrix}
  c^2 & -ct_2 \cr
  -c\wt{t}_2 & c^2  \cr
 \end{pmatrix}
=
 \begin{pmatrix}
  c_{+-} & c_{--} \cr
  c_{++} & c_{+-}  \cr
 \end{pmatrix}^{-1},
\]
we have, for any boundary condition $P$,
\[
\brak{P}=I(P).
\]
\end{prop}

\begin{figure}[htp]
\centering
\includegraphics[width=\textwidth]{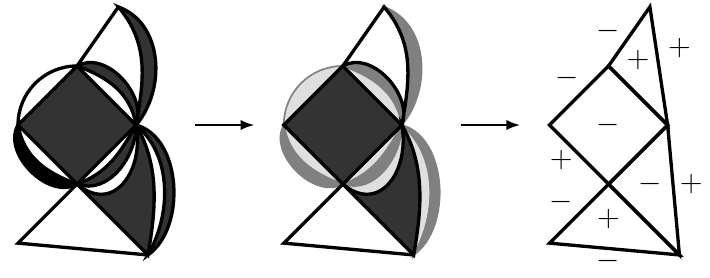}
\caption{Illustration of the many-to-one correspondence, from hypermaps to Ising-decorated maps, through the contraction of digons.}
\label{fig:bic-to-ising} 
\end{figure}

\begin{proof}
Starting from a hypermap $\mathfrak{m} \in  \textnormal{Bic}^{d,\widetilde{d}}(P)$, let us contract its digons (see Figure~\ref{fig:bic-to-ising}) and interpret the colors of the remaining faces as spins. This is a many-to-one correspondence from $\textnormal{Bic}^{d,\widetilde{d}}(P)$ to $\textnormal{I}_{d,\wt d}(P)$. If two hypermaps $\mathfrak{m}_1, \mathfrak{m}_2$  yield the same Ising map $\mathfrak{m}$, necessarily both $\mathfrak{m}_1$ and $\mathfrak{m}_2$ are obtained from $\mathfrak{m}$ by inserting on each edge a sequence of digons that alternate in color, and such that the initial and final digons have the color opposite of the face of $\mathfrak{m}$ that they are glued to. Therefore, for each edge of $\mathfrak{m}$, both the number of white and black digons in both sequences must have the same parity. This implies that
\[
\brak{P}=\sum_{\mathfrak{m}\in\textnormal{I}_{d,\wt d}(P)}t^{v(\mathfrak{m})}\prod_{3\leq i\leq d}t_i^{f_{+}^i(\mathfrak{m})}\prod_{3\leq j\leq \widetilde{d}}\widetilde{t}_j^{f_{-}^j(\mathfrak{m})}\left(\frac{c\wt t_2}{c^2-t_2\wt t_2} \right)^{e_{++}(\mathfrak{m})}\left(\frac{c t_2}{c^2-t_2\wt t_2} \right)^{e_{--}(\mathfrak{m})}\left(\frac{c^2}{c^2-t_2\wt t_2} \right)^{e_{+-}(\mathfrak{m})},
\]
which matches $I(P)$ provided that 
\[ 
\begin{pmatrix}
  c^2 & -ct_2 \cr
  -c\wt{t}_2 & c^2  \cr
 \end{pmatrix}
=
 \begin{pmatrix}
  c_{+-} & c_{--} \cr
  c_{++} & c_{+-}  \cr
 \end{pmatrix}^{-1}.
\]
\end{proof}

\subsection{Tutte decompositions and splitting procedure}
\label{sec:tutte as splitting}

We now prove a Tutte-type decomposition for the generating function $\brak{P}$ for any given boundary condition $P$. We call such a decomposition, a \emph{splitting procedure}, once again by analogy with matrix models, see for instance~\cite{eynard-formal-matrix,eynard}.

\begin{prop}
\label{prop:splitting-procedure}
Let $P$ be a word over $\{A,B\}$, then
\[
c\brak{PA} =- \brak{P\widetilde{V}'(B)} + \sum_{\substack{B \textnormal{ in } P \\\textnormal{s.t. }P=P_1 B P_2}}\brak{P_1}\brak{P_2}.
\]
\end{prop}

\begin{proof}
Consider a hypermap $\mathfrak{m}$ with boundary condition $PA$, contributing to 
\[
 \Big[t^{v}c^{-e}\prod_{2\leq i\leq d}t_i^{f_\circ^i}\prod_{2\leq j\leq \widetilde{d}}\widetilde{t}_j^{f_\bullet^j}\Big]\brak{PA},
\] and peel the edge associated to $A$ on the boundary of $\mathfrak{m}$  (see Figure~\ref{fig:splitting}). If it was incident to an inner black face of degree $k$, the map $\mathfrak{m}'$ obtained by removing that face has one less inner black face of degree $k$, one less edge, and $k-1$ more boundary edges of type $B$, that are inserted in the boundary condition where $A$ was. Thus, $\mathfrak{m}'$ contributes to 
\[
\Big[t^{v}c^{-e+1}\prod_{2\leq i\leq d}t_i^{f_\circ^i}\cdot\wt{t}_k^{f_\bullet^k-1}\prod_{j\neq k}\wt{t}_j^{f_\bullet^j}\Big]\brak{PB^k}.
\]
If $A$ was identified with a boundary edge of type $B$, then $\mathfrak{m}$ is made of two smaller maps $\mathfrak{m}_1$, $\mathfrak{m}_2$ connected by a bridge, and their respective boundary conditions $P_1, P_2$ must satisfy $P=P_1BP_2$. Moreover, their number of edges sum up to $e-1$, while all the inner faces of $\mathfrak{m}$ appear in either $\mathfrak{m}_1$ or $\mathfrak{m}_2$ . Therefore, the pair $(\mathfrak{m}_1$, $\mathfrak{m}_2)$ contributes to 
\[
  \Big[t^{v}c^{-e+1}\prod_{2\leq i\leq d}t_i^{f_\circ^i}\prod_{2\leq j\leq \widetilde{d}}\widetilde{t}_j^{f_\bullet^j}\Big]\brak{P_1}\brak{P_2}.
\]
Summing over all possible maps, and as 
\[
\sum_{2\leq k\leq \wt{d}}\wt{t}_kB^k=-\wt V'(B),
\]
we obtain the announced equation.
\end{proof}

\begin{figure}[htp]
\centering
\includegraphics[width=\textwidth]{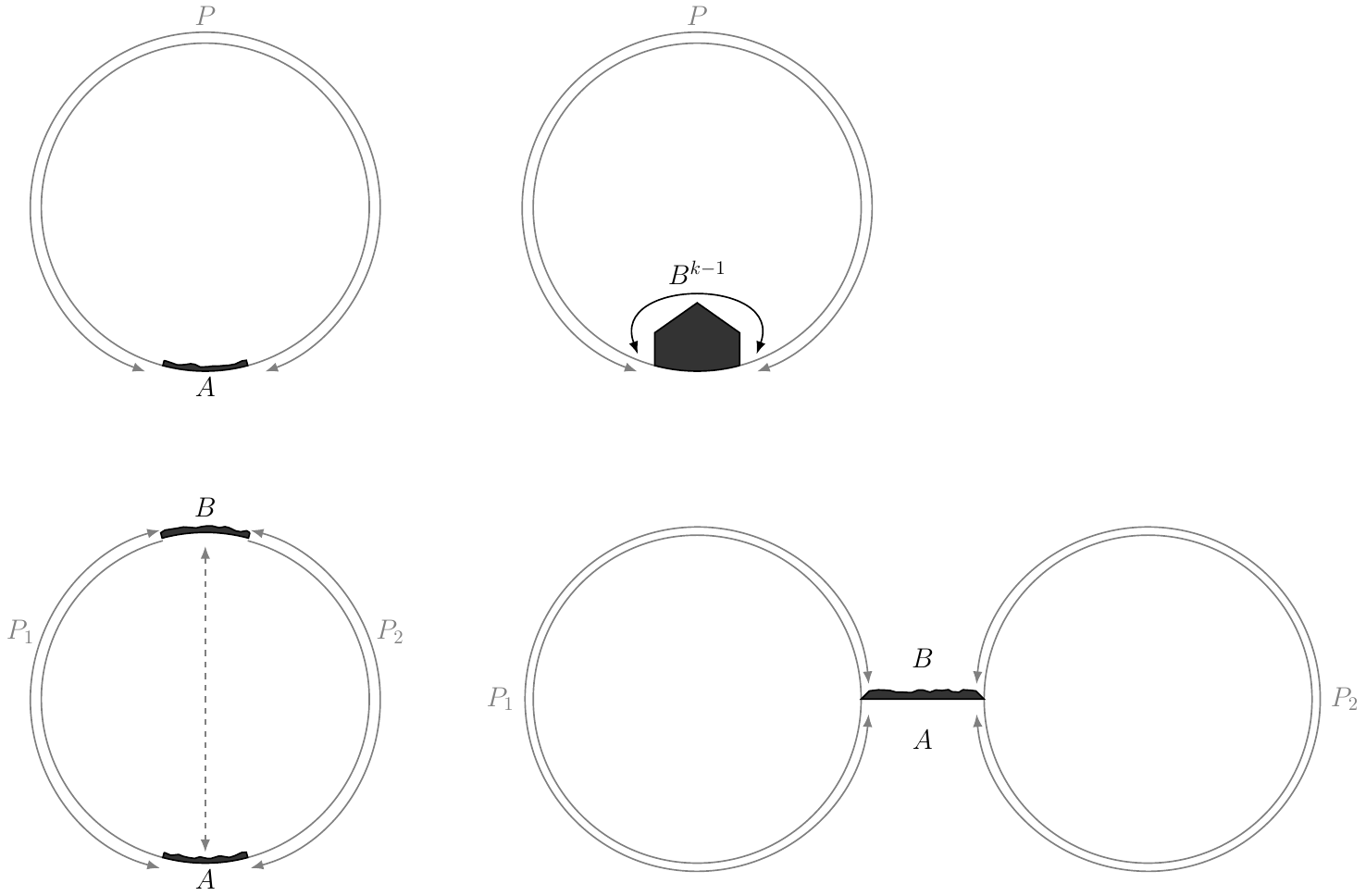}
\caption{A sketch of the splitting procedure: top left, the boundary edge of type $A$ that we peel. Top right, the case where it is adjacent to a black inner face. Bottom row, the case where it is identified with a boundary edge of type $B$.}
\label{fig:splitting} 
\end{figure}

From our definition of $\brak{\cdot}$, this extends linearly to any formal non-commutative power series in $A,B$: if $P=\sum_kc_k\w_k$, then
\begin{equation}
c\brak{PA} =\mathbf{-} \brak{P\widetilde{V}'(B)} + \sum_{k}c_k\sum_{\substack{B \textnormal{ in } \w_k \\\textnormal{s.t. }\w_k=P_1 B P_2}}\brak{P_1}\brak{P_2}.
\end{equation}

We also have the symmetric equation for the splitting of $\brak{PB}$.

\begin{rem}
Note that $P_1$ or $P_2$ in the sum can be $1$. Also, note that the sum can be empty (if there is no power of $B$ in $P$).
\end{rem}

\section{Master equation for the alternating boundary condition}
\label{sec:general-master-eq}

\subsection{Splitting equations}

We will now derive two equations on the function $M$, involving the other auxiliary generating functions, by applying the splitting procedure to various series.

\begin{prop}
The generating function $M$ satisfies the following two equations:
\begin{equation}
    c(y - Y)(x\widetilde{S}_+ - \widetilde{R}) = M\left(cx(\widetilde{V}'+cx-f_{10})+\widetilde{R}\omega\right) - c^2xf_{00} + x\widehat{P} + \widetilde{R}(-V'-cy+f_{01}),
\label{eq:first-equation-in-M}
\end{equation}
\begin{equation}
    M(\widehat{f}\omega + cxY) = -xR + \widehat{f}(W-f_{01}).
    \label{eq:second-equation-in-M}
\end{equation}
\end{prop}
 
\begin{proof}
Let us first write
\begin{align*}
    &\brak{\frac{1}{x-A}\frac{1}{\omega - BA}B\frac{\widetilde{V}'(y)-\widetilde{V}'(B)}{y-B}}= \brak{\frac{1}{x-A}\frac{1}{\omega - BA}(\widetilde{V}'(y)-\widetilde{V}'(B))\left(\frac{y}{y-B}-1\right)} \\
    &= y\brak{\frac{1}{x-A}\frac{1}{\omega - BA}\frac{\widetilde{V}'(y)-\widetilde{V}'(B)}{y-B}} - \brak{\frac{1}{x-A}\frac{1}{\omega - BA}(\widetilde{V}'(y)-\widetilde{V}'(B))} \\
    &= y\widetilde{S}_+ - \widetilde{V}'M + \brak{\frac{1}{x-A}\frac{1}{\omega - BA}\widetilde{V}'(B)} \\
    &=  y\widetilde{S}_+ - \widetilde{V}'M - \left(c\brak{\frac{1}{x-A}\frac{1}{\omega - BA}A} - \brak{\frac{1}{x-A}\frac{1}{\omega - BA}}\brak{\frac{1}{\omega - BA}A} \right)\\
&\gr{\Bigg[\uparrow\textnormal{Splitting of }\brak{\frac{1}{x-A}\frac{1}{\omega - BA}\mathbf{A}}\Bigg]}\\ 
    &= y\widetilde{S}_+ - \widetilde{V}'M - c\brak{\left(\frac{x}{x-A}-1\right)\frac{1}{\omega - BA}} + Mf_{10} = y\widetilde{S}_+ - \widetilde{V}'M - cxM + cf_{00} + Mf_{10}.
\end{align*}

On the other hand:
\begin{align*}
 &c\brak{\frac{1}{x-A}\frac{1}{\omega - BA}B\frac{\widetilde{V}'(y)-\widetilde{V}'(B)}{y-B}} \\
    &=\brak{\frac{1}{x-A}\frac{1}{\omega - BA}V'(A)\frac{\widetilde{V}'(y)-\widetilde{V}'(B)}{y-B}} + \brak{\frac{1}{x-A}}\brak{\frac{1}{x-A}\frac{1}{\omega - BA}\frac{\widetilde{V}'(y)-\widetilde{V}'(B)}{y-B}} \\
    &\hspace{4em}+ \brak{\frac{1}{x-A}\frac{1}{\omega - BA}B}\brak{\frac{1}{\omega - BA}\frac{\widetilde{V}'(y)-\widetilde{V}'(B)}{y-B}} \\
&\gr{\Bigg[\uparrow\textnormal{Splitting of } c\brak{\frac{1}{x-A}\frac{1}{\omega - BA}\mathbf{B}\frac{\widetilde{V}'(y)-\widetilde{V}'(B)}{y-B}}\Bigg]}\\ 
    &= -V'(x)\widetilde{S}_++\widehat{P} + W\widetilde{S}_+ +\brak{\frac{1}{x-A}\frac{1}{\omega - BA}B}\widetilde{R} \\
    &= -V'(x)\widetilde{S}_++\widehat{P}+W\widetilde{S}_++\brak{\frac{1}{x}\left(1+\frac{A}{x-A}\right)\frac{1}{\omega - BA}B}\widetilde{R} \\
    &= -V'(x)\widetilde{S}_++\widehat{P}+W\widetilde{S}_++\frac{\widetilde{R}}{x}\left(f_{01} + \brak{\frac{1}{x-A}\frac{1}{\omega - BA}BA}\right) \\
    &= -V'(x)\widetilde{S}_++\widehat{P}+W\widetilde{S}_++\frac{\widetilde{R}}{x}\left(f_{01} + \brak{\frac{1}{x-A}\left(\frac{\omega}{\omega - BA}-1\right)}\right) \\
    &= -V'(x)\widetilde{S}_++\widehat{P}+W\widetilde{S}_++\frac{\widetilde{R}}{x}\left(f_{01} + \omega M - W\right).
\end{align*}

Combining the two previous equations, we obtain:
\begin{align*}
    c\left(y\widetilde{S}_+ - \widetilde{V}'M - cxM + cf_{00} + Mf_{10}\right) = -V'(x)\widetilde{S}_++\widehat{P}+W\widetilde{S}_++\frac{\widetilde{R}}{x}\left(f_{01} + \omega M - W\right),
\end{align*}
which can be rewritten as
\begin{align*}
    \widetilde{S}_+(cy + V' - W) = M\left(c\widetilde{V}'+c^2x-cf_{10}+\frac{\widetilde{R}\omega}{x}\right) - c^2f_{00} + \widehat{P} + \frac{\widetilde{R}}{x}(f_{01}-W),
\end{align*}
which yields equation~\eqref{eq:first-equation-in-M}. 

Let us now turn to~\eqref{eq:second-equation-in-M}. We have:
\begin{align*}
&c\brak{\frac{1}{x-A}\frac{1}{\omega-BA}B} \\
&= -\brak{\frac{1}{x-A}\frac{1}{\omega-BA}V'(A)} + \brak{\frac{1}{x-A}}\brak{\frac{1}{x-A}\frac{1}{\omega-BA}} + \brak{\frac{1}{x-A}\frac{1}{\omega-BA}B}\brak{\frac{1}{\omega-BA}}\\
&\gr{\Bigg[\uparrow\textnormal{Splitting of } c\brak{\frac{1}{x-A}\frac{1}{\omega-BA}\mathbf{B}}\Bigg]}\\
&=- V'M + R + WM +\brak{\frac{1}{x-A}\frac{1}{\omega-BA}B}f_{00},
\end{align*}
so that
\begin{align*}
    \brak{\frac{1}{x-A}\frac{1}{\omega-BA}B} = -\frac{1}{\widehat{f}}(R+cYM).
\end{align*}
On the other hand,
\begin{align*}
    &\brak{\frac{1}{x-A}\frac{1}{\omega-BA}B} = \brak{\frac{1}{x}\left(1+\frac{A}{x-A}\right)\frac{1}{\omega-BA}B} = \frac{1}{x}\left(f_{01} + \brak{\frac{1}{x-A}\frac{1}{\omega-BA}BA}\right) \\
    &= \frac{1}{x}\left(f_{01} + \brak{\frac{1}{x-A}\left(\frac{\omega}{\omega-BA}-1\right)}\right)= \frac{1}{x}\left(f_{01}+\omega M - W\right).
\end{align*}

Combining the two previous equations, we obtain:
\begin{equation*}
    x(R + cYM) = -\widehat{f}(f_{01}+\omega M - W),
\end{equation*}
which yields~\eqref{eq:second-equation-in-M}.

\end{proof}

From equations~\eqref{eq:first-equation-in-M} and~\eqref{eq:second-equation-in-M}, we deduce the following:
\begin{prop}
Let 
\begin{equation}
Q(x,y) :=-\frac{\widehat{f}}{c}\Big(V'+cy -f_{01}-\frac{xR}{\widehat{f}}\Big)\Big(\widetilde{V}'+cx -f_{10}-\frac{y\widetilde{R}}{\widehat{f}}\Big)+ \frac{1}{c^2}(\widehat{f}\omega +cxy)\big(c^2f_{00}-\widehat{P} + \frac{R\widetilde{R}}{\widehat{f}}\big).
\label{eq:Q}
\end{equation}
Then $Q(x,Y(x))=0$.
\label{prop:Q=0}
\end{prop}

\begin{proof}
Substituting $y$ by $Y(x)=-V'(x)+O(x^{-1})$ in~\eqref{eq:first-equation-in-M}, the left hand side vanishes, so that the right hand side is equal to zero as well, giving one expression for $M$:
\begin{equation*}
M = \frac{c^2xf_{00} - x\widehat{P} + \widetilde{R}(V'+cy-f_{01})}{cx(\widetilde{V}'+cx-f_{10})+\widetilde{R}\omega}\lvert_{y=Y(x)}. 
\end{equation*}

Equation~\eqref{eq:second-equation-in-M} gives a second expression for $M$:
\begin{equation*}
M=\frac{-xR + \widehat{f}(+V'+cY-f_{01})}{\widehat{f}\omega + cxY}.
\end{equation*}

Therefore, the following polynomial in $x$ and $y$
\begin{equation*}
\left(c^2xf_{00} - x\widehat{P} + \widetilde{R}(V'+cy-f_{01})\right)\left(\widehat{f}\omega + cxy\right) - \left(cx(\widetilde{V}'+cx-f_{01})+\widetilde{R}\omega\right)\left(-xR + \widehat{f}(V'+cy-f_{01})\right)
\end{equation*}
vanishes for $y=Y(x)$. Dividing it by $xc^2$, we get:
\begin{equation*}
\begin{aligned}
\frac{1}{c^2}\Big(&- c \widehat{f} (V'+cy-f_{01})(\widetilde{V}'+cx -f_{10}) + cxR(\widetilde{V}'+cx -f_{01}) \\
&+ cy\widetilde{R}(V' +cy -f_{01}) +R\widetilde{R}\omega + (\widehat{f}\omega +cxy)(c^2f_{00}-\widehat{P})\Big),
\end{aligned}
\end{equation*}
which is equal to $Q(x,y)$.
\end{proof}

\subsection{Algebraicity of $\wh{f}$}
\label{sec: algebraicity-2-strats}

In this section, we establish a strategy to derive an algebraic equation on $\wh{f}(\omega)$ from Proposition~\ref{prop:Q=0}, and compare it with the kernel method that is detailed in Appendix~\ref{sec:appendix-alg-by-bmj}.

\begin{prop}
We have
\begin{equation}
Q(x,y)=E(x,y).
\label{eq:Q=E}
\end{equation}
\label{prop:Q=E}
\end{prop}

\begin{proof}
 Both $Q(x,y)$ and $E(x,y)$ are polynomials in $\Q(c)[t_i,\wt t_j,\omega,f_{ij}][[t]][x,y]$, and both of them have degree $d$ in $x$ and $\wt d$ in $y$. Moreover, both of them vanish when $y$ is substituted by $Y(x)$ (and likewise when $x$ is substituted by $X(y)$, by symmetry). Since $E$ is irreducible (this is clear from the fact that the associated algebraic curve $\{(x(z),y(z))\, | \, z \in \mathbb{CP}^1\}$ is connected), necessarily this implies that $Q$ is equal to $E$ up to a multiplicative factor in $\QQ(c,t_i,\wt t_j,\omega,f_{ij})((t))$. To identify this proportionality constant, let us look at the highest degree coefficient in $y$ in both polynomials:
\begin{itemize}
\item the highest degree monomial in $y$ in $E$ comes from the term $cy\widetilde{V}'$ and is equal to $-c\widetilde{t}_{\widetilde{d}}y^{\tilde{d}}$.
\item The highest degree monomial in $y$ in $Q$ comes from the terms $-\widehat{f}y\widetilde{V}'+\widetilde{R}y^2$, and is equal to $\widehat{f}\widetilde{t}_{\widetilde{d}}y^{\widetilde{d}} - f_{00}\widetilde{t}_{\widetilde{d}}y^{\widetilde{d}} =-c\widetilde{t}_{\widetilde{d}}y^{\tilde{d}}$ as well.
\end{itemize}
\vspace{-1em}\end{proof}

Using Proposition~\ref{prop:Q=E}, we can prove our general algebraicity result:

\begin{thm}
\label{thm:algebraicity}
The coefficients of monomials in $x$ and $y$  of the polynomial $Q(x,y)-E(x,y)$ induce a system of $(d-1)(\wt d -1)$ polynomials $C_{i,j} \in \Q(c)[t_i,\wt t_j,\omega][[t]][g_{0,0}, \dots, g_{d-2,0},g_{1,1}, \dots, g_{d-2,\wt d -2}], \, 0 \leq i \leq d-2,0\leq j\leq \wt{d}-2$, that satisfy $C_{i,j}(f_{0,0}, \dots, f_{d-2,0},f_{1,1}, \dots, f_{d-2,\wt d -2})=0$ for all $i,j$, and whose Jacobian determinant $\det J=\det(\frac{\partial C_{i,j}}{\partial g_{k\ell}})$ evaluated at $(f_{0,0},\dots,f_{d-2,\wt{d}-2})$ is a non-zero element of $\Q(c)[t_i,\wt t_j,\omega]$. Consequently, all the series $f_{ij}$, and in particular $\wh{f}=f_{00}-c$, are algebraic over $\QQ(c,t_i,\wt t_j,\omega)((t))$.
\end{thm}

\begin{proof}
Note that the coefficients in $y^{\wt d -1}$ in $E$ and $Q$ are also trivially equal, as both are equal to $-c\wt{t}_{\wt{d}-2}-\wt{t}_{\wt{d}-1}V'(x)$ (and likewise for $x^{d-1}$). Thus, we have indeed $(d-1)(\wt d -1)$ polynomial relations: $[x^iy^j](Q-E)$, $0\leq i \leq d-2$, $0\leq j \leq \wt d -2$, that we denote by $C_{i,j}$. Let us also write $Q_{i,j}$ for $[x^iy^j]Q(x,y)$. We first expand $R$ and $\wh{P}$ in $x$ and $y$:
\[
R=-\sum_{0\leq i \leq d-2}x^i\sum_{i+2\leq k \leq d}t_kf_{k-2-i,0}, \qquad \wh{P}=\sum_{\substack{0\leq i \leq d-2\\ 0\leq j \leq \wt{d}-2}}x^iy^j\sum_{\substack{i+2\leq k \leq d\\ j+2\leq \ell\leq \wt d}}t_k\wt{t}_\ell f_{k-2-i,\ell-2-j}.
\]
Hence, for $1\leq i \leq d-2, 1\leq j\leq \wt d -2, (i,j)\neq (d-2,\wt{d}-2), (1,1)$,
\[
Q_{i, j}=-\frac{\wh{f}}{c}t_{i+1}\wt{t}_{j+1}+\frac{1}{c^2}\sum_{\substack{i+2\leq k \leq d\\ j+2\leq \ell\leq \wt d}}t_k\wt{t}_\ell\left( \omega f_{k-2-i,0}f_{0,\ell-2-j}-\omega\wh{f} f_{k-2-i,\ell-2-j}-cf_{k-1-i,\ell-1-j}\right),
\]
and, for $1\leq i \leq d-2$,
\[
Q_{i,0}=\frac{\wh{f}}{c}t_{i+1}f_{1,0}-ct_i+\frac{1}{c}\sum_{i+1\leq k\leq d}t_k(f_{1,0}f_{k-i-1,0}-cf_{k-i,0}) +\frac{\omega}{c^2}\sum_{\substack{i+2\leq k \leq d\\ 2\leq \ell \leq \wt d}} t_k\wt{t}_\ell \left(f_{k-2-i,0}f_{0,\ell-2}-\wh{f} f_{k-2-i,\ell-2}\right),
\]
and finally
\[
Q_{d-2, \wt{d}-2}=t_{d-1}\wt{t}_{\wt d -1}+\frac{\omega}{c}t_d\wt{t}_{\wt{d}}f_{0,0}-\frac{1}{c}t_d\wt{t}_{\wt{d}}f_{1,1},
\]
\[
Q_{1, 1}=-\frac{\wh{f}}{c}t_{2}\wt{t}_{2}+c^2+\frac{1}{c^2}\sum_{\substack{3\leq k \leq d\\ 3\leq \ell\leq \wt d}}t_k\wt{t}_\ell\left( \omega f_{k-3,0}f_{0,\ell-3}-\omega\wh{f} f_{k-3,\ell-3}-cf_{k-2,\ell-2}\right),
\]
\[
Q_{0, 0}=-\frac{\wh{f}}{c}f_{1,0}f_{0,1}+\wh{f}\omega f_{0,0}+\frac{\omega}{c^2}\sum_{\substack{2\leq k \leq d\\ 2\leq \ell\leq \wt d}}t_k\wt{t}_\ell\left(f_{k-2,0}f_{0,\ell-2}-\wh{f} f_{k-2,\ell-2}\right).
\]
Note that from their definition, all the series $f_{i,j}$ with $(i,j)\neq (0,0)$ are power series in $c^{-1}$ without constant term, while $\wh{f}=-c+t/\omega+g$, where $g$ is a power series in $c^{-1}$ without constant term. Hence, $(\det J)(f_{i,j})$ is a Laurent series in $c$ with a finite upper bound on its powers. To show that it is non-zero, it therefore suffices to check that its highest degree coefficient in $c$ is non-zero. Let us now look at dependence in $c$ of the elements of $J$. For $i,j,k,\ell\geq 1$, we have (with the convention that $t_m=0$, $t_n=0$ if the indices $m$, $n$ are not within the correct ranges):
\begin{align*}
\frac{\partial C_{i,j}}{\partial f_{k,\ell}}&=-\frac{1}{c^2}\omega\wh{f}t_{k+i+2}\wt{t}_{\ell+j+2}-\frac{1}{c}t_{k+i+1}\wt{t}_{\ell+j+1}=\frac{1}{c}(t_{k+i+2}\wt{t}_{\ell+j+2}\omega-t_{k+i+1}\wt{t}_{\ell+j+1})+O(c^{-2}),\\
\frac{\partial C_{i,j}}{\partial f_{k,0}}&=O(c^{-1}), \qquad \frac{\partial C_{i,0}}{\partial f_{k,0}}=-t_{k+i}(1+\delta_{(i,1)})+O(c^{-1}), \qquad \frac{\partial C_{0,j}}{\partial f_{k,0}}=O(c^{-1}),\\
\frac{\partial C_{i,j}}{\partial f_{0,0}}&=-\frac{1}{c}t_{i+1}\wt{t}_{j+1}+O(c^{-2}), \qquad \frac{\partial C_{i,0}}{\partial f_{0,0}}=O(c^{-2}), \qquad \frac{\partial C_{0,0}}{\partial f_{0,0}}=-c\omega+O(1).
\end{align*}
Thus, the contribution to $(\det J)(f_{i,j})$ corresponding to picking the elements 
\[
\frac{\partial C_{d-i-1,\wt{d}-j-1}}{\partial f_{i,j}} \text{ for } i,j \neq 0, \, \frac{\partial C_{d-i-1,0}}{\partial f_{i,0}} \text{ for } i \geq 1  \ \text{ and likewise for the } f_{0,j}, \ \ \text{ and } \frac{\partial C_{0,0}}{\partial f_{0,0}},
\]
is equal to 
\[
(-1)^{d\wt{d}+d+\wt{d}+1}(t_d)^{(d-2)(\wt d -1)}(\wt t_{\wt d})^{(d-1)(\wt d -2)}c^{1-(d-2)(\wt d -2)}+O(c^{-(d-2)(\wt d -2)}),
\]
and every other term in $(\det J)(f_{i,j})$ is $O(c^{-(d-2)(\wt d -2)})$. This concludes the proof.
\end{proof}

Theorem~\ref{thm:algebraicity} gives a strategy to obtain an algebraic equation on $\wh{f}(\omega)$ that does not rely on the kernel method, contrary to the strategy starting from~\eqref{eq:second-equation-in-M} detailed in Appendix~\ref{sec:appendix-alg-by-bmj}. These two strategies differ conceptually, and they also entail different computational steps. Let us conclude this subsection by comparing them on these two fronts.

In terms of computational steps, the kernel method (see Appendix~\ref{sec:appendix-alg-by-bmj} for details and notation) generically requires the elimination of $2d$ auxiliary variables (the $M_i$'s and the $\bar{Z}_i$), and involves (among other steps) the identification of the  $\bar{Z}_i$'s as the $d$ formal power series solutions of an equation of degree $d+\wt d$. Our strategy generically requires the elimination of $(d-1)(\wt d -1)-1$ auxiliary series (the $f_{i,j}$ with $(i,j)\neq (0,0)$). It is straightforward to inductively eliminate the $f_{i,j}$ with $i,j \neq 0$ using the equation $C_{d-i-1,d-j-1}$, and the resulting polynomial system on the remaining auxiliary series $f_{i,0},f_{0,j}$ is comprised of $d+\wt{d}-4$ quadratic equations in as many variables. In full generality, it is not clear to say which strategy is the simplest. However, in the special case of Ising quadrangulations (see Section~\ref{sec:ising-quad}), we can use the parity of the potentials to simplify both strategies. With the strategy of Theorem~\ref{thm:algebraicity}, we only need to solve a linear system on 3 variables ($f_{22},f_{20},f_{02}$). Compared to this, the kernel method, while simpler than in the generic case, still requires solving a quartic degree equation on $(\bar{Z}_i)^2$, which is much more complicated.

Another computational point of comparison is the amount of information required on the monochromatic generating function $Y(x)$. While we use the known explicit rational parametrization of $(x,Y(x))$ in Section~\ref{sec:ising-quad}, in principle we only need the explicit expression of $E(x,y)$ and the initial coefficients $D_{ij}$, while the kernel method requires the whole rational parametrization, since it is necessary to be able to  write an equation on $M$ (or $\mathcal{M}$) that is polynomial in the catalytic variable. In particular, our method might be generalized to other problems where the equivalent of $Y(x)$ is known to be algebraic, with an explicit annihilating polynomial, but where there is no known explicit parametrization.\\

On a more conceptual level, our strategy uses two catalytic variables $x$ and $y$, corresponding to two different non-alternating parts of the boundary. This might seem more complicated to eliminate than the single catalytic variable of the kernel method, since there is no general framework to eliminate several catalytic variables from a functional equation. However, these two catalytic variables are ``automatically eliminated'' in a sense, when we turn to the coefficients in $x$ and $y$ of $Q-E$. This is possible in particular because $x,y$ do not only correspond to refinements of the enumeration problem associated to $\wh{f}(\omega)$, but have an independent combinatorial meaning in the enumeration of the monochromatic boundary condition.

Another notable feature of our strategy is that it keeps the $x\leftrightarrow y$ symmetry of the problem, which is clearly not the case of the kernel method applied to $M$.

\subsection{Re-derivation of the Tutte equation and kernel system}
\label{sec:bouttier-carrance-eq-rederived}

Recall from Section~\ref{subsec:constell} that equation~\eqref{eq:bouttier-carrance-eq-notation} (equation (27) in~\cite{bouttier-carrance})  is the starting point of applying the kernel method to $\wh{f}$, as mentioned in~\cite{bouttier-carrance} and detailed here in Appendix~\ref{sec:appendix-alg-by-bmj}. The following result relates it to the splitting equations that we have derived:

\begin{prop}
Equation~\eqref{eq:second-equation-in-M} implies equation~\eqref{eq:bouttier-carrance-eq-notation} (equation (27) in~\cite{bouttier-carrance}).
\label{prop:tutte-bc-from-splitting}
\end{prop}

\begin{proof}
We can rewrite the right hand side $\ul R(x,w)$ of~\eqref{eq:bouttier-carrance-eq-notation} using a splitting equation:
\begin{align*}
\ul R(x,w)  &= cw x W(x)Y(x) + w\left[xV'(x)\frac{M(\frac{1}{w},x)}{w}\right]_{x^{\geq 0}}= cw x W(x)Y(x)  + \brak{\frac{xV'(x) - AV'(A)}{(x-A)(1/w-BA)}} \\ 
&= cwxW(x)Y(x)  + xR(x,1/w)  -\brak{\frac{AV'(A)-xV'(A)}{(x-A)(1/w-BA)}} \\ 
&= cwxW(x)Y(x)  + xR(x,1/w)  + \brak{\frac{V'(A)}{1/w-BA}} \\ 
&= cwxW(x)Y(x)  + xR(x,1/w) - \left(c\brak{\frac{1}{1/w - BA}B} - \brak{\frac{1}{1/w-BA}B}\brak{\frac{1}{1/w-BA}}\right)\\ 
&\gr{\Bigg[\uparrow\textnormal{Splitting of } c\brak{\frac{1}{1/w - BA}\mathbf{B}}\Bigg]},
\end{align*} 
i.e.
\begin{equation}
\ul R(x,w)= cwxW(x)Y(x)  + xR(x,1/w) + \widehat{f}(1/w)f_{10}(1/w).
\label{eq:remainder-bc-expression}
\end{equation}

Since equation~\eqref{eq:second-equation-in-M}, with $\omega = 1/w$, is equivalent to
\[
\left(-\frac{\widehat{f}(\frac{1}{w})}{w} - cxY(x)\right)M\left(\frac{1}{w},x\right) =   xR\left(x,\frac{1}{w}\right) - \widehat{f}\left(\frac{1}{w}\right)W(x) + \widehat{f}\left(\frac{1}{w}\right)f_{01}\left(\frac{1}{w}\right),
\]
it implies that
\[
\left(c-f_{00}\left(\frac{1}{w}\right) - cwxY(x)\right)\left(\frac{M(\frac{1}{w},x)}{w}-W(x)\right) = cwxW(x)Y(x) + xR\left(x,\frac{1}{w}\right) + \widehat{f}\left(\frac{1}{w}\right)f_{01}\left(\frac{1}{w}\right),
\]
which gives equation (27) in~\cite{bouttier-carrance} by specializing $c$ to 1.
\end{proof}

\begin{rem}
Equations~\eqref{eq:second-equation-in-M} and~\eqref{eq:bouttier-carrance-eq-notation} are both Tutte-like equations of the same generating function $M$, but here compared to~\cite{bouttier-carrance} we considered more general boundary conditions with extra $B$'s to be able to write~\eqref{eq:second-equation-in-M}.
\end{rem}

Let us also note that from our master equation~\eqref{eq:Q=E}, we can also obtain a ``symmetrized'' version of the kernel system $\{K=0,\ul R=0\}$ of~\cite{bouttier-carrance} for the unique formal power series $\mathpzc{w}(\xi)$ satisfying the kernel relation~\eqref{eq:kernel-relation}. Indeed, substituting $\omega$ by $\Omega(\xi)=\mathpzc{w}(\xi)^{-1}$ in $Q$, we obtain:
\begin{align*}
Q(x,y)&_{\lvert \omega=\Omega(xy)}=-\frac{\wh{f}(\Omega)}{c}\Big(V'(x)+cy-f_{01}(\Omega)-\frac{xR(\Omega)}{\widehat{f}}\Big)\Big(\widetilde{V}'(y)+cx -f_{10}-\frac{y\widetilde{R}(\Omega)}{\widehat{f}(\Omega)}\Big)\\
&=\frac{-1}{c\wh{f}(\Omega)}\left(-\frac{cxyV'}{\Omega}-\frac{c^2xy^2}{\Omega}-f_{01}(\Omega)\wh{f} - xR(\Omega)\right)\cdot\left(-\frac{cxy\wt V'}{\Omega}-\frac{c^2x^2y}{\Omega}-f_{10}(\Omega)\wh{f} - y\wt R(\Omega)\right)\\
&=\frac{-1}{c\wh{f}(\Omega)}\left(cxy\mathpzc{w}(xy)(V'+cy)+f_{01}(\Omega)\wh{f}+ xR(\Omega)\right)\cdot \left(cxy\mathpzc{w}(xy)(\wt V'+cx)+f_{10}(\Omega)\wh{f}+ y\wt R(\Omega)\right).
\end{align*}
Therefore, performing the additional substitution $x=x(z), y=y(z)$, we obtain, using~\eqref{eq:remainder-bc-expression} (with the slight abuse of notation $\mathpzc{w}(z)=\mathpzc{w}(x(z)y(z))$, and after specialization of $c$ to 1, to fit with~\cite{bouttier-carrance}):
\[
0=Q(x(z),y(z))_{\lvert \omega=\Omega(z)}=\frac{-1}{c\wh{f}(\Omega)}\ul R(x(z),\mathpzc{w}(z))\cdot\wt{\ul R}(y(z),\mathpzc{w}(z)),
\]
so that $\mathpzc{w}(z)$ must satisfy either $\ul R(x(z),\mathpzc{w}(z))=0$ or $\wh{\ul R}(y(z),\mathpzc{w}(z))=0$.

Note that this gives two choices for the second equation of the system on $(\mathpzc{w}(z), A(\mathpzc{w}(z))$ in terms of $z$, but (in cases where we do not need to eliminate $z$, e.g. bicolored $m$-angulations) these should lead to the same expression for the solution in terms of $\xi=x(z)y(z)$, as we know that $\mathpzc{w}(\xi)$ is determined uniquely as a formal power series in $\xi^{-1}$.

\section{Explicit solution for Ising quadrangulations}
\label{sec:ising-quad}

In this section, we take $V(x) = -\frac{t_4}{4}x^4 - \frac{t_2}{2}x^2$ and $\widetilde{V}(y) =- \frac{\widetilde{t}_4}{4}y^4 - \frac{\widetilde{t}_2}{2}y^2$, which corresponds to quadrangulations decorated by the Ising model as explained in Section~\ref{sec:corresp-ising}, that we call \emph{Ising quadrangulations} for short. We apply to this choice of potentials the general strategy of Theorem~\ref{thm:algebraicity}.

\subsection{Parametrization for the monochromatic boundary}

With that choice of $V$ and $\wt{V}$, the parametrization of~\cite[Theorem 8.3.1]{eynard} is written as
\begin{equation}
\begin{aligned}
x(z) &= \gamma z + \alpha_1 z^{-1} +\alpha_3 z^{-3} \\
y(z) &= \gamma z^{-1} + \beta_1 z + \beta_3 z^3,
\label{eq:parametrization-of-x(z)-and-y(z)-for-E=0}
\end{aligned}
\end{equation}
and the system~\eqref{eq:system-on-alpha-beta-gamma} characterizing the functions $\alpha_i,\beta_j,\gamma$ can be written as (see the \textsc{SageMath} companion file for details)
\begin{equation}
\begin{aligned}
-c\beta_3 = \gamma^3t_4&, \qquad -c\alpha_3 = \gamma^3\widetilde{t}_4,\\
 -c\beta_1 = 3t_4\gamma^2\alpha_1+t_2\gamma&, \qquad -c\alpha_1 = 3\widetilde{t}_4\gamma^2\beta_1 +\widetilde{t}_2\gamma,\\
-\frac{1}{c} &= -\gamma^2 + \alpha_1\beta_1 + 3\alpha_3\beta_3.
\label{eq:system-for-alpha-i-beta-i-for-t2-and-t4}
\end{aligned}
\end{equation}
These relations allow us to express the original variables $t,t_2,t_4,\wt{t}_2,\wt{t}_4$, as rational functions of the parameters $\alpha_1,\alpha_3,\beta_1,\beta_3,\gamma$, and the original variable $c$:
\begin{equation}
\begin{aligned}
t &= c(\gamma^2-\alpha_1\beta_1-3\alpha_3\beta_3),\\
t_2&=\frac{c}{\gamma^2}(\beta_1\gamma-3\alpha_1\beta_3), \qquad t_4=\frac{c\beta_3}{\gamma^3},\\
\wt{t}_2&=\frac{c}{\gamma^2}(\alpha_1\gamma-3\beta_1\alpha_3), \qquad \wt{t}_4=\frac{c\alpha_3}{\gamma^3}.
\end{aligned}
\label{eq:param-t-ti-ising-quads}
\end{equation}

\subsection{Explicit master equation for $\wh{f}$}
\label{sec:master-eq-ising-quads}

We apply the general strategy of Theorem~\ref{thm:algebraicity} to Ising quadrangulations, and use the above parametrization~\eqref{eq:param-t-ti-ising-quads}, to obtain the following:
\begin{prop}
In the case of Ising quadrangulations, the pair $(\omega, \wh{f}(\omega))$ satisfies an explicit polynomial equation with coefficients in $\Q[c, \gamma, \alpha_1, \alpha_3, \beta_1,\beta_3]$, of degree 3 in $\wh{f}$ and 7 in $\omega$, and the associated algebraic curve is of genus 0. Moreover, the pair of functions
\[
\zeta(\omega)=\omega+\frac{\gamma^6}{\alpha_3\beta_3\omega}, \qquad \lambda(\omega)=\wh{f}(\omega)
\]
satisfies an explicit polynomial equation with coefficients in $\Q[c, \gamma, \alpha_1, \alpha_3, \beta_1,\beta_3]$, of degree 3 in both $\zeta$ and $\lambda$.
\label{prop:eqs-ising-q-non-sym-final}
\end{prop}

To show Proposition~\ref{prop:eqs-ising-q-non-sym-final}, we first obtain an equation on $\omega, \wh{f}$ whose coefficients depend on the initial terms $D_{ij}$ of the Dobrushin boundary condition:

\begin{lem}
The pair $(\omega, \wh{f}(\omega))$ satisfies:
\[
\textnormal{Eq}(\wh{f},\omega)=0,
\]
where
\begin{equation}
\begin{aligned}
&\textnormal{Eq}(\wh{f},\omega)=(-t_4^3\wt{t}_4^3\omega^7 + 2c^2t_4^2\wt{t}_4^2\omega^5 - c^4t_4\wt{t}_4\omega^3)\wh{f}^3+\Bigg(-ct_4^3\wt{t}_4^3w^7 + t\wt{t}_4^3\wt{t}_4^3w^6 + c(c^2 - t_2\wt{t}_2)t_4^2\wt{t}_4^2\omega^5 \\
&- c^2t_4\wt{t}_4(\wt{t}_2^2t_4 + t_2^2\wt{t}_4 + 2tt_4\wt{t}_4)\omega^4 + c^3(c^2 -t_2\wt{t}_2)t_4\wt{t}_4\omega^3 + c^4tt_4\wt{t}_4\omega^2 - c^7\omega   \Bigg)\wh{f}^2\\
+& \Bigg(-(D_{11}ct_4^3\wt{t}_4^3 + (c^2 + t_2\wt{t}_2)c^2t_4^2\wt{t}_4^2)\omega^5 + (-c^2(t_4\wt{t}_2^2 + \wt{t}_4t_2^2) +  D_{20}t_2t_4\wt{t}_4^2 +  D_{02}\wt{t}_2t_4^2\wt{t}_4 + 2tt_2\wt{t}_2t_4 \wt{t}_4)ct_4\wt{t}_4\omega^4 \\
&+ 2(tt_2^2\wt{t}_4+t\wt{t}_2^2t_4+D_{02}\wt{t}_2t_4\wt{t}_4+D_{20}\wt{t}_2t_4\wt{t}_4+cD_{11}t_4\wt{t}_4-c^2t_2\wt{t}_2t_4\wt{t}_4+c^4)c^2t_4\wt{t}_4\omega^3 \\
&+ (2tt_2\wt{t}_2t_4\wt{t}_4+D_{02}t_2t_4\wt{t}_4^2+D_{20}\wt{t}_2t_4^2\wt{t}_4-c^2\wt{t}_2^2t_4-c^2t_2^2\wt{t}_4)c^3\omega^2 - (c^3 + c^t_2\wt{t}_2 + D_{11}t_4\wt{t}_4)c^5\omega\Bigg)\wh{f}\\
+&(D_{02}t_2\wt{t}_4+D_{20}\wt{t}_2t_4+D_{22}t_4\wt{t}_4+tt_2\wt{t}_2-c^2t)c^2t_4^2\wt{t}_4^2\omega^4+\Big(D_{02}c^2\wt{t}_2t_4\wt{t}_4+D_{20}c^2t_2t_4\wt{t}_4-D_{02}tt_2t_4\wt{t}_4^2\\
&-D_{20}t\wt{t}_2t_4^2\wt{t}_4-D_{02}D_{20}t_4\wt{t}_4+c^2tt_2^2\wt{t}_4+c^2t\wt{t}_2^2t_4-t^2t_2\wt{t}_2t_4\wt{t}_4-c^4t_2\wt{t}_2\Big)ct_4\wt{t}_4\omega^3\\
+&\Big(2c^2tt_2\wt{t}_2t_4\wt{t}_4+2c^4tt_4\wt{t}_4-t^2\wt{t}_2^2t_4^2\wt{t}_4-t^2t_2^2t_4\wt{t}_4^2-c^4\wt{t}_2^2t_4-c^4t_2\wt{t}_4-D_{02}^2t_4^2\wt{t}_4^3-D_{20}^2t_4^3\wt{t}_4^2\\
&-2D_{02}t\wt{t}_2t_4^2\wt{t}_4^2-2D_{20}tt_2t_4^2\wt{t}_4^2-2D_{22}c^2t_4^2\wt{t}_4^2\Big)c^2\omega^2\\
+&\Big(c^2t\wt{t}_2^2t_4+c^2tt_2^2\wt{t}_4+c^2D_{02}\wt{t}_2t_4\wt{t}_4+c^2D_{20}t_2t_4\wt{t}_4-c^4t_2\wt{t}_2-D_{02}tt_2t_4\wt{t}_4^2-D_{20}t\wt{t}_2t_4^2\wt{t}_4\\
&-D_{02}D_{20}t_4^2\wt{t}_4^2-t^2t_2\wt{t}_2t_4\wt{t}_4\Big)c^3\omega +c^6(tt_2\wt{t}_2+D_{02}t_2\wt{t}_4+D_{20}\wt{t}_2t_4+D_{22}t_4\wt{t}_4-c^2t).
\end{aligned}
\label{eq:master-equation-Ising-quadrangulations-dij}
\end{equation}
\end{lem}

\begin{proof}
To obtain~\eqref{eq:master-equation-Ising-quadrangulations-dij}, we start from the system of equations $[x^ky^\ell]Q=[x^ky^\ell]E$, $0 \leq k \leq d-2$, $0 \leq \ell \leq \wt d -2$, and eliminate the auxiliary functions $f_{ij}$ for $(i,j)\neq(0,0)$. Let us start by writing down the explicit expressions of the functions $R, \wt{R}, \wh{P}$ that appear in $Q$, and $P$ that appears in $E$. We have
\[
R(x) = -t_4(x^2f_{00} +xf_{10}+f_{20}) - t_2f_{00}, \qquad \widetilde{R}(y) =- \widetilde{t}_4(y^2f_{00} +yf_{01}+f_{02}) - \widetilde{t}_2f_{00}.
\]
Note that $f_{11}=\braket{\frac{1}{\omega - BA}BA}=\omega f_{00}-t$ (we could rederive this from the coefficient of $x^2y^2$ in $Q-E$, but it is quicker to write it immediately). Thus we have for $\wh{P}$:
\begin{align*}
\widehat{P}(x,y) =& (t_4x^2+t_2)(\widetilde{t}_4y^2+\widetilde{t}_2)f_{00} +t_4x(\widetilde{t}_4y^2+\widetilde{t}_2)f_{10} +\widetilde{t}_4y(t_4x^2+t_2)f_{01} +t_4\widetilde{t}_4xy(\omega f_{00} - t) \\
&+t_4(\widetilde{t}_4y^2+\widetilde{t}_2)f_{20} +\widetilde{t}_4(t_4x^2+t_2)f_{02} +t_4\widetilde{t}_4xf_{12} +t_4\widetilde{t}_4yf_{21} +t_4\widetilde{t}_4f_{22}.
\end{align*}
We express $P$ in terms of the functions $D_{ij}=\brak{B^iA^j}$ similarly to $\widehat{P}$ (note that $D_{11}$ cannot immediately be expressed in a simpler way, contrary to $f_{11}$):
\begin{align*}
P(x,y) =& x^2y^2t_4\widetilde{t}_4D_{00} +x^2yt_4\widetilde{t}_4D_{10} +xy^2t_4\widetilde{t}_4D_{01} +xyt_4\widetilde{t}_4D_{11}\\
& +x^2t_4(\widetilde{t}_4D_{02}+\widetilde{t}_2D_{00}) +y^2\widetilde{t}_4(t_4D_{20}+t_2D_{00}) \\
&+xt_4(\widetilde{t}_4D_{12}+\widetilde{t}_2D_{10}) +y\widetilde{t}_4(t_4D_{21}+t_2D_{01}) +(t_4\widetilde{t}_4D_{22}+t_4\widetilde{t}_2D_{20}+t_2\widetilde{t}_4D_{02}+t_2\widetilde{t}_2D_{00}).
\end{align*}
We then use the parity of the potentials, which implies that $D_{ij} = f_{ij} = 0$ if $i+j$ is odd, so that
\[
R(x) = -x^2t_4f_{00} -(t_4f_{20} + t_2f_{00}), \qquad \widetilde{R}(y) =- y^2\widetilde{t}_4f_{00} -(\widetilde{t}_4f_{02} + \widetilde{t}_2f_{00}),
\]
\begin{align*}
\widehat{P}(x,y) =& x^2y^2(t_4\widetilde{t}_4f_{00}) +xy(t_4\widetilde{t}_4(\omega f_{00}-t)) +x^2(t_4\widetilde{t}_2f_{00}+t_4\widetilde{t}_4f_{02}) +y^2(t_2\widetilde{t}_4f_{00}+t_4\widetilde{t}_4f_{20}) \\
&+(t_2\widetilde{t}_2f_{00}+t_4\widetilde{t}_2f_{20}+t_2\widetilde{t}_4f_{02}+t_4\widetilde{t}_4f_{22}),
\end{align*}
and
\begin{align*}
P(x,y) &= x^2y^2(t_4\widetilde{t}_4D_{00}) +xy(t_4\widetilde{t}_4D_{11}) +x^2t_4(\widetilde{t}_4D_{02}+\widetilde{t}_2D_{00}) +y^2\widetilde{t}_4(t_4D_{20}+t_2D_{00}) \\
&+(t_4\widetilde{t}_4D_{22}+t_4\widetilde{t}_2D_{20}+t_2\widetilde{t}_4D_{02}+t_2\widetilde{t}_2D_{00}).
\end{align*}
We then plug these into the expressions for the respective expressions~\eqref{eq:E-def} for $E$ and~\eqref{eq:Q} for $Q$. As detailed in the \textsc{SageMath} companion file, we then consider the system of equations on $\{\omega, \wh{f},f_{20},f_{02},f_{22}\}$ induced by stating that the coefficients in $x^iy^j$ of $Q-E$ must be zero. The equations corresponding to the coefficients in $x^2$, $y^2$ and $xy$ allow us to eliminate $f_{20},f_{02},f_{22}$. Writing down the constant coefficient in $Q(x,y)-E(x,y)$, we are left with a polynomial equation in $\widehat{f}$ and $\omega$, with coefficients in $\Q[c,t, t_i,\wt{t}_j, D_{20},D_{02},D_{11},D_{22}]$: $0=\textnormal{Eq}(\wh{f},\omega)$, where $\textnormal{Eq}(\wh{f},\omega)$ is given in the statement of the proposition.
\end{proof}

\begin{proof}[Proof of Proposition~\ref{prop:eqs-ising-q-non-sym-final}]
In equation~\eqref{eq:master-equation-Ising-quadrangulations-dij}, the functions $D_{ij}$ are for the moment non-explicit functions in $(c,t, t_k, \wt{t}_\ell)$. We now make use of the rational parametrization~\eqref{eq:parametrization-of-x(z)-and-y(z)-for-E=0} and the expression~\eqref{eq:param-t-ti-ising-quads} of $(t, t_k, \wt{t}_\ell)$ in terms of the parameters $(\alpha_m,\beta_n, \gamma, c)$, to write the $D_{ij}$ themselves as explicit rational functions (and actually, polynomials) in $(\alpha_m,\beta_n, \gamma, c)$, through a Lagrange inversion formula. Indeed, from~\eqref{eq:parametrization-of-x(z)-and-y(z)-for-E=0}, denoting $\bar{x}:=x^{-1}$ and $\bar{z}:=z^{-1}$, we have
\begin{equation}
\begin{aligned}
\bar{z} &=\bar{x}\left(\gamma+\alpha_1\bar{z}^2+\alpha_3\bar{z}^4\right)=:\bar{x}\Phi(\bar{z})\\
y(z) &=\gamma\bar{z}+\beta_1\bar{z}^{-1}+\beta_3\bar{z}^{-3}=:\Upsilon(\bar{z}),
\end{aligned}
\end{equation}
so that, by the Lagrange--Bürmann formula, for any non-negative $k$, the coefficient of ${x}^{-(k+1)}$ in $Y(x)$ is
\begin{equation}
[\bar{x}^{k+1}]Y(x)=\frac{1}{k+1}[u^k]\left(\Upsilon'(u)\Phi(u)^{k+1}\right).
\label{eq:lagrange-inv-x-z}
\end{equation}

From this, we directly obtain expressions for $D_{20}=[x^{-3}]Y(x)$ and (by symmetry) $D_{20}$:
\begin{equation}
\begin{aligned}
D_{20}&=c\left(\alpha_1\gamma^3-\alpha_3\beta_1\gamma^2-6\alpha_1\alpha_3\beta_3\gamma-\alpha_1^2\beta_1\gamma-\alpha_1^3\beta_3\right),\\
D_{02}&=c\left(\beta_1\gamma^3-\beta_3\alpha_1\gamma^2-6\beta_1\beta_3\alpha_3\gamma-\beta_1^2\alpha_1\gamma-\beta_1^3\alpha_3\right).
\end{aligned}
\label{eq:M02-M20-param-ratio}
\end{equation}

For $D_{11}$ and $D_{22}$, we make use of the splitting procedure to express them as combinations of functions $D_{i0}$:
\begin{equation*}
\begin{aligned}
D_{11}&=\brak{AB}=\frac{1}{c}\left(t^2+\wt{t}_2D_{20}+\wt{t}_4D_{40}\right),\\
D_{22}&=\brak{A^2B^2}=\frac{1}{c^2}\left(t^3+\wt t_4D_{20}^2+3t\wt t_2D_{20}+(\wt t_2^2+3t\wt t_4)D_{40}+2\wt t_2\wt t_4D_{60}+\wt t_4^2D_{80}\right),
\end{aligned}
\end{equation*}
so that, applying~\eqref{eq:lagrange-inv-x-z}, we get:
\begin{equation}
\begin{aligned}
D_{11}=&c\left(\gamma^4-\alpha_1\beta_1\gamma^2-5\alpha_3\beta_3\gamma^2-(\alpha_1^2\beta_3+\beta_1^2\alpha_3)\gamma+3\alpha_3^2\beta_3^2-\alpha_1\beta_1\alpha_3\beta_3\right),\\
D_{22}=&c\big(\gamma^6-6\alpha_3\beta_3\gamma^4-2(\alpha_1^2\beta_3+\beta_1^2\alpha_3)\gamma^3+(4\alpha_3^2\beta_3^2-11\alpha_1\beta_1\alpha_3\beta_3-\alpha_1^2\beta_1^2)\gamma^2\\
     &-(2\alpha_1^2\alpha_3\beta_3^2+2\beta_1^2\alpha_3^2\beta_3+\alpha_1^3\beta_1\beta_3+\alpha_1\beta_1^3\alpha_3)\gamma-\alpha_1^2\beta_1^2\alpha_3\beta_3+2\alpha_1\beta_1\alpha_3^2\beta_3^2-\alpha_3^3\beta_3^3\big).
\end{aligned}
\label{eq:M11-M22-param-ratio}
\end{equation}

Injecting~\eqref{eq:param-t-ti-ising-quads},~\eqref{eq:M02-M20-param-ratio} and~\eqref{eq:M11-M22-param-ratio} into~\eqref{eq:master-equation-Ising-quadrangulations-dij}, we get an explicit polynomial equation $\textnormal{Eq}_2(\omega,\wh{f})$ of degree 7 in $\omega$ and 3 in $\wh{f}$, with coefficients that are now in $\Q[c,\alpha_1,\beta_1,\alpha_3,\beta_3,\gamma]$. As it is a bit long, and straightforward to obtain from~\eqref{eq:master-equation-Ising-quadrangulations-dij} as explained, we do not write it here, but it is available in the \textsc{SageMath} companion file. 

The polynomial $\textnormal{Eq}_2(\omega,\wh{f})$ greatly simplifies when we perform the following transformation:
\begin{equation}
\left(\omega,\wh{f}\right) \mapsto \left(\zeta=\omega+\frac{\gamma^6}{\alpha_3\beta_3\omega}, \lambda=\wh{f}\omega\right).
\end{equation}

We first straightforwardly convert the equation from $(\omega,\wh{f})$ to $(\omega,\lambda=\wh{f}\omega)$, by taking the numerator $\textnormal{Eq}_3(\omega,\lambda)$ of $\textnormal{Eq}_2(\omega,\frac{\lambda}{\omega})$. This already decreases the degree in $\omega$ from 7 to 6. For $\zeta$, we take the resultant of $\textnormal{Eq}_3(\omega,\lambda)$ and $D=\alpha_3\beta_3\omega^2-\alpha_3\beta_3\zeta\omega+\gamma^6$. This resultant is actually the square of a polynomial in $\zeta$ and $\lambda$ with coefficients in $\Q[c,\alpha_1,\beta_1,\alpha_3,\beta_3,\gamma]$, that we take to be our equation $\textnormal{Eq}_4(\zeta,\lambda)$ for these new variables. It has degree 3 in both $\zeta$ and $\lambda$.
\end{proof}

We have checked using \textsc{Maple} that $\textnormal{Eq}_2(\omega,\wh{f})=0$ is of genus 0 (and a fortiori so is $\textnormal{Eq}_4(\zeta,\lambda)$)\footnote{We used \textsc{Maple} for this because the \texttt{genus} command of \textsc{SageMath} is only implemented for curves over number fields.}. However, even for the simpler polynomial $\textnormal{Eq}_4(\zeta,\lambda)$, it is not possible to obtain immediately a rational parametrization with a command such as \texttt{algcurves[parametrization]} in \textsc{Maple}, due to the presence of many (six) parameters. It should be possible to obtain an explicit rational parametrization in a ``bruteforce'' manner using interpolation over the parameters. We set aside this question for now, and focus on the special case of \emph{symmetric} Ising quadrangulations (corresponding to the specialization $t_i=\wt{t}_i$), where we obtain an explicit rational parametrization of $(\omega, \wh{f})$ with minimal use of computer algebra.

\subsection{The symmetric case}
\label{sec:rat-param-ising-quads}

We now derive the explicit rational parametrization for the case of symmetric Ising quadrangulations, announced in Theorem~\ref{thm:ratio-param-Ising-quad-sym}

\begin{proof}[Proof of Theorem~\ref{thm:ratio-param-Ising-quad-sym}]
All our parameters our now symmetric:
\[
\alpha_3 = \beta_3 , \quad \alpha_1 = \beta_1, \quad t_4 = \widetilde{t}_4 = \frac{c \alpha_3}{\gamma^3} , \quad t_2 = \widetilde{t}_2 = c\alpha_1\frac{\gamma - 3\alpha_3}{\gamma^2}, \quad t= c(\gamma^2-\alpha_1^2-3\alpha_3^2),
\]
so that the $D_{ij}$ are equal to
\begin{align*}
D_{00} &= c(\gamma^2 - \alpha_1^2 - 3\alpha_3^2) =t \nonumber, \\
D_{02} &= D_{20} = c(\gamma^3 \alpha_1  - \gamma^2 \alpha_1 \alpha_3 - \gamma \alpha_1^3 - 6 \gamma \alpha_3^2 \alpha_1 - \alpha_3 \alpha_1^3) \nonumber, \\
D_{22} &= c(\gamma^6 - 6\gamma^4\alpha_3^2 - 4\gamma^3\alpha_1^2\alpha_3 + \gamma^2(4\alpha_3^4 - \alpha_1^4 - 11\alpha_1^2\alpha_3^2) \nonumber - \gamma(\alpha_1^4\alpha_3 + 4\alpha_1^2\alpha_3^3) - \alpha_1^4\alpha_3^2 + 2\alpha_1^2\alpha_3^4 - \alpha_3^6).
\end{align*}
As detailed in the \textsc{SageMath} companion file, the polynomial $\textnormal{Eq}_{4,\textnormal{sym}}(\zeta,\lambda)$ then factorizes into
\[
\textnormal{Eq}_{4,\textnormal{sym}}(\zeta,\lambda)=\alpha_3(2\gamma^3+\alpha_3\zeta)\textnormal{Eq}_{\textnormal{sym}}(\zeta,\lambda),
\]
where $\textnormal{Eq}_{\textnormal{sym}}(\zeta,\lambda) \, \in \Q[c,\alpha_1,\alpha_3,\gamma][\zeta,\lambda]$ is equal to:
\begin{align}
\textnormal{Eq}_{\textnormal{sym}}(\zeta,\lambda)=&\zeta^2c\alpha_3^2\lambda^2+\zeta\Big(\alpha_3^4\lambda^3+(\alpha_1^2\alpha_3^4+3\alpha_3^6-\alpha_3^4\gamma^2-2\alpha_3^3\gamma^3)c\lambda^2 \nonumber\\ 
+&(3\alpha_3^8-\alpha_1^2\alpha_3^6-2\alpha_1^2\alpha_3^5\gamma+8\alpha_1^2\alpha_3^4\gamma^2-5\alpha_3^6\gamma^2-6\alpha_1^2\alpha_3^3\gamma³+\alpha_1^2\alpha_3^3\gamma^4+\alpha_3^4\gamma^4+\alpha_3^2\gamma^6)c^2\lambda \nonumber \\ \nonumber
+&(\alpha_1^4\alpha_3^6 - 2\alpha_1^2\alpha_3^8 + \alpha_3^1 - 4\alpha_1^4\alpha_3^5\gamma + 4\alpha_1^2\alpha_3^7\gamma + 6\alpha_1^4\alpha_3^4\gamma^2 + 2\alpha_1^2\alpha_3^6\gamma^2 - 4\alpha_3^8\gamma^2 - 4\alpha_1^4\alpha_3^3\gamma^3 \\ \nonumber
&- 8\alpha_1^2\alpha_3^5\gamma^3 + \alpha_1^4\alpha_3^2\gamma^4 + 2\alpha_1^2\alpha_3^4\gamma^4 + 6\alpha_3^6\gamma^4 + 4\alpha_1^2\alpha_3^3\gamma^5 - 2\alpha_1^2\alpha_3^2\gamma^6 - 4\alpha_3^4\gamma^6 + \alpha_3^2\gamma^8)c^3\Big)\\ \nonumber
-&2\alpha_3\gamma^3\lambda^3+(9\alpha_1^2\alpha_3^4\gamma^2-8\alpha_1^2\alpha_3^3\gamma^3-6\alpha_3^5\gamma^3+\alpha_1^2\alpha_3^2\gamma^4+2\alpha_3^3\gamma^5)c\lambda^2\\ \nonumber
+&(-6\alpha_1^4\alpha_3^5\gamma+14\alpha_1^4\alpha_3^4\gamma^2+18\alpha_1^2\alpha_3^6\gamma^2-10\alpha_1^4\alpha_3^3\gamma^3-28\alpha_1^2\alpha_3^5\gamma^3-6\alpha_3^7\gamma^3+2\alpha_1^4\alpha_3^2\gamma^4\\ \nonumber
&+12\alpha_1^2\alpha_3^3\gamma^5+10\alpha_3^5\gamma^5-2\alpha_1^2\alpha_3^2\gamma^6-2\alpha_3^3\gamma^7-2\alpha_3\gamma^9)c^2\lambda\\ \nonumber
+&(\alpha_1^6\alpha_3^6 - 4\alpha_1^6\alpha_3^5\gamma - 6\alpha_1^4\alpha_3^7\gamma + 6\alpha_1^6\alpha_3^4\gamma^2 + 20\alpha_1^4\alpha_3^6\gamma^2 + 9\alpha_1^2\alpha_3^8\gamma^2 - 4\alpha_1^6\alpha_3^3\gamma^3 \\ \nonumber
&- 20\alpha_1^4\alpha_3^5\gamma^3 - 20\alpha_1^2\alpha_3^7\gamma^3 - 2\alpha_3^9\gamma^3 + \alpha_1^6\alpha_3^2\gamma^4 - 4\alpha_1^2\alpha_3^6\gamma^4 + 10\alpha_1^4\alpha_3^3\gamma^5 + 36\alpha_1^2\alpha_3^5\gamma^5\\ \nonumber
& + 8\alpha_3^7\gamma^5 - 4\alpha_1^4\alpha_3^2\gamma^6 - 18\alpha_1^2\alpha_3^4\gamma^6 - 12\alpha_1^2\alpha_3^3\gamma^7 - 12\alpha_3^5\gamma^7 + 12\alpha_1^2\alpha_3^2\gamma^8 - 4\alpha_1^2\alpha_3\gamma^9\\
& + 8\alpha_3^3\gamma^9 + \alpha_1^2\gamma^{10} - 2\alpha_3\gamma^{11})c^3.
\label{eq:master-equation-Ising-quadrangulations-symmetric}
\end{align}

Since it is of degree $2$ in $\zeta$, it will correspond to a rational curve if it satisfies the so-called ``one-cut property''. In other words, we want to check whether, when considering~\eqref{eq:master-equation-Ising-quadrangulations-symmetric} as an equation in $\zeta$ with coefficients depending on $\lambda$, the associated discriminant $\Delta(\lambda)$ (which is a polynomial in $\lambda$), only has 2 simple zeros, and the rest of its zeros are of even degree. Indeed, this means that $\Delta(\lambda)$ can be written as
\[
\Delta(\lambda)=F(\lambda)^2(\lambda-a)(\lambda-b),
\]
where $F$ is a polynomial in $\lambda$. Then, we automatically get a joint rational parametrization of $\lambda$ and $\zeta$ satisfying~\eqref{eq:master-equation-Ising-quadrangulations-symmetric}, via Zhukovsky's transformation:
\[
\lambda(H):=\frac{a+b}{2}+\frac{a-b}{4}\left(H+\frac{1}{H}\right),
\]
since we then have (for a given choice of the square root)
\[
\sqrt{(\lambda-a)(\lambda-b)}=\frac{a-b}{4}\left(H-\frac{1}{H}\right),
\]
so that the two solutions $\zeta_{\pm}$ of~\eqref{eq:master-equation-Ising-quadrangulations-symmetric} are also rational functions of $H$. Of these two solutions, only one has an asymptotic behavior when $H\to\infty$ that is compatible with the initial definitions of $\lambda$ and $\zeta$ as Laurent series in $\omega$ (that satisfy $\lambda,\zeta \to \infty, \lambda/\zeta\to -c$ as $\omega \to \infty$), so that there is no ambiguity in writing an expression for $\zeta(H)$. Note that, since there is a symmetry $\lambda(\frac{1}{H})=\lambda(H)$, $\zeta_{\pm}(\frac{1}{H})=\zeta_{\mp}(H)$, the choice of looking at the behavior as $H\to\infty$ and not $H\to 0$ does not ultimately impact the formulae for $\lambda$ and $\zeta$. 

These verifications and computations are straightforward to do with a computer algebra system (see the \textsc{SageMath} companion file), and we obtain the following rational parametrization of the algebraic curve associated to $\textnormal{Eq}_{\textnormal{sym}}(\zeta,\lambda)$:
\begin{equation}
\begin{aligned}
\zeta_\textnormal{sym}(H) =& \frac{1}{(H\alpha_1 + \alpha_3 + \gamma)^2\alpha_3^2}\Big(H^3(\alpha_1^3\alpha_3^3 + \alpha_1^3\alpha_3^2\gamma)+2H^2\gamma\alpha_1^2\alpha_3^2 (3\alpha_3 + 2\gamma)\\
&+H\alpha_1\gamma^2(9\alpha_3^3+ 7\alpha_3^2\gamma - \alpha_3\gamma^2 + \gamma^3)+\alpha_3\gamma^3(2\alpha_3^2 + 4\alpha_3\gamma + 2\gamma^2)\Big)\\
\lambda_\textnormal{sym}(H) =& -\frac{c}{H}\left(H\alpha_1+\alpha_3+\gamma\right)\left(H(\alpha_3+\gamma)+\alpha_1\right).
\end{aligned}
\label{eq:parametrization-for-Ising-quadrangulations-symmetric-lambda-zeta}
\end{equation}
To get to a parametrization for our initial variables $(\omega,\widehat{f})$, note that the equation on $(H,\omega)$ given by $\zeta(\omega)=\zeta_\textnormal{sym}(H)$ is itself of degree 2 in $\omega$, and we check that it also satisfies the one-cut property (see the \textsc{SageMath} companion file for details), so that we find a joint parametrization $(H(h),\omega(h))$ for $H$ and $\omega$ satisfying this relation given by $\zeta$. Since $\widehat{f}$ is a rational function of $\omega$ and $\lambda$, and $\lambda(H(h))$ is also a rational function of $h$, this yields a joint rational parametrization $(\omega(h),\widehat{f}(h))$. As before, there is ultimately no importance in the choice of a solution to the quadratic equation in $\omega$, since the change $h\to 1/h$ leads to $\omega_{\pm}\to \omega_{\mp}$ and leaves $\lambda$ unchanged. We thus obtain the announced rational parametrization of $(\omega,\wh{f})$:
\begin{equation*}
\begin{aligned}
\omega_\textnormal{sym}(h) &= \frac{(\alpha_3 h +\gamma)^2\gamma^3h}{(\gamma h + \alpha_3)^2\alpha_3}\\
\widehat{f}_\textnormal{sym}(h) &=-c\frac{(\alpha_3 \gamma h^2 + h(\alpha_1^2 - 2\alpha_3\gamma) + \alpha_3 \gamma) (\gamma h +\alpha_3)^3}{(\alpha_3 h + \gamma) \gamma^4 (h - 1)^2h^2}.
\end{aligned}
\end{equation*}
\end{proof}

\begin{rem}
One can check that specializing $\alpha_1 = 0$ in~\eqref{eq:parametrization-for-Ising-quadrangulations-symmetric-f-omega} gives back the parametrization of~\cite{bouttier-carrance} for pure bicolored quadrangulations, with $h = z^4 = s$.
\end{rem}

As explained in Section~\ref{subsec:constell}, a remarkable feature of the parametrization obtained for $(\omega,\widehat{f})$ in the case of $m$-constellations in~\cite{bouttier-carrance}, is that $\omega$ and $\widehat{f}$ are written as rational functions of $x,Y(x)$, and inherit from this a parametrization in $z$: thus, $x,Y(x)$ and $\omega,\widehat{f}(\omega)$ can be written jointly as rational functions of a single variable $z$, while satisfying the \emph{kernel relation} $\widehat{f}\omega+cxY(x)=0$. 

Here, in the case of Ising quadrangulations, we could not obtain a parametrization of $(\omega,\widehat{f})$ so directly. Furthermore, the explicit solution of this new special case allows us to prove Corollary~\ref{cor:non-co-rat}:

\begin{proof}[Corollary~\ref{cor:non-co-rat}]
It obviously suffices to prove that the property does not hold for the case of symmetric Ising quadrangulations. Substituting $\omega_\textnormal{sym}(h),\widehat{f}_\textnormal{sym}(h), x(z),y(z)$ into the kernel relation gives a polynomial equation on $h$ and $z$. We show that the associated algebraic curve is non-rational. We actually obtain something even stronger. Consider the algebraic equation on $H$ and $Z$ given by $\lambda_{\textnormal{sym}}(H)=-c\xi(Z)$, with $\xi(Z)=x(z)y(z)_{|Z=z^2}$ (which is a rational function of $Z$). It can be written as
\[
cZ^2\left(H\alpha_1+\alpha_3+\gamma\right)\left(H(\alpha_3+\gamma)+\alpha_1\right)=H\left(\gamma Z^2+\alpha_1Z+\alpha_3\right)\left(\gamma+\alpha_1Z+\alpha_3Z\right).
\]
In particular, it is  a polynomial of degree 2 in $H$, whose discriminant has 4 distinct simple poles in $Z$ (see the \textsc{SageMath} companion file for details). Therefore, it cannot be of genus 0. \emph{A fortiori}, the relation $\lambda(H(h))=-c\xi(z^2)$ cannot correspond to a genus 0 curve. Since both rational parametrizations $(x(z),y(z))$ and $(\omega_{\textnormal{sym}}(h),\wh{f}_{\textnormal{sym}}(h))$ are proper, this implies that the two curves $(\omega,\wh{f})$, $(x,Y(x))$ do not admit a joint rational parametrization of the form (with some abuse of notation) $(x(S),y(S),\omega(S),\wh{f}(S))$ satisfying the relation $\omega(S)\wh{f}(S)=-cx(S)y(S)$. 
\end{proof}

The question of looking for a parametrization of $\omega,\wh{f}$ as rational functions of $x,y$ (i.e. ``constructing  $(\omega,\wh{f})$ on the curve $E(x,y)$'') without requiring the kernel relation is left open, and would be related to broader questions on the genus of variable-separated curves (see for instance~\cite{pakovich}). However, we believe that the kernel relation is a combinatorially significant relation. In particular, it also plays an important role in the slice decomposition of the alternating boundary condition~\cite{slices-alt}.

\section{Conclusion and perspectives}
\label{sec:ccl}

In this paper, we have established a new strategy to obtain an explicit algebraic master equation on the generating functions of general hypermaps with an alternating boundary, that proves to be much more efficient than the classical kernel method, at least in the case of Ising quadrangulations. This represents significant progress in the study of the alternating condition. There are still many open questions regarding this combinatorial problem.

First, we do not yet have an explicit formula for the master equation for general potentials. Given the concrete examples that we have at hand, it is natural to conjecture that the associated curve is always of genus 0. If it is true, it would be satisfying to have a general expression for a rational parametrization, like in the monochromatic case.

As mentioned in Section~\ref{subsec:constell}, we also expect the generating functions for general topologies to satisfy the topological recursion. If it is indeed the case, it would be interesting to understand in particular how the nice relation between the monochromatic and alternating conditions for $m$-constellations evolve in higher topologies, and if there is any relation between the two spectral curves in the general case. 

As a related question, it would be interesting to have a better understanding of the kernel relation, beyond its computational importance. Another computational step that we would like to make sense of more conceptually is the operation of taking the coefficients in $x$ and $y$ of the equation $Q=E$, as this is a crucial step in our strategy. It might also give some insight on the relation between the monochromatic and alternating conditions.

Finally, to further the combinatorial understanding of hypermaps,  it is also natural to want to obtain both the known and expected formulae by bijective arguments. We will address this in future work~\cite{slices-alt}, using the slice decomposition method.

\section*{Acknowledgements}

This work was initially supported by the ERC-SyG project, Recursive and Exact New Quantum Theory (ReNewQuantum) which received funding from the European Research Council (ERC) under the European Union's Horizon 2020 research and innovation programme under grant agreement No 810573.
In particular this project was born during the problem session of the Otranto summer school organized by ReNewQuantum. AC is currently supported by the Austrian Science Fund (FWF) grant 10.55776/F1002. The authors thank Jérémie Bouttier and Thomas Lejeune for insight and valuable discussions.

\appendix

\section{Application of the kernel method to general potentials}
\label{sec:appendix-alg-by-bmj}

In this appendix, we will detail how to write the Tutte equation~\eqref{eq:bouttier-carrance-eq-notation} obtained in~\cite{bouttier-carrance}, in a way that fits in the following general theorem:

\begin{thm}[Theorem 3 in~\cite{BoJe06}]
Let $Q(y_0,y_1,\dots,y_k,t,v)$ be a polynomial in $k+3$ indeterminates, with coefficients in a field $\K$. We consider the functional equation
\begin{equation}
F(u)\equiv F(t,u)=F_0(u)+tQ(F(u),\Delta F(u), \Delta^{(2)}F(u), \dots, \Delta^{(k)}F(u),t,u),
\label{eq:bm-j-thm}
\end{equation}
where $F_0(u) \in \K[u]$ is given explicitly and the operator $\Delta$ is the divided difference:
\[
\Delta F(u)=\frac{F(u)-F(0)}{u},
\]
and $\Delta^{(i)}$ is its $i$-th iterate. Then~\eqref{eq:bm-j-thm} has a unique solution $F(t,u) \in \K[u][[t]]$, and this formal power series is algebraic over $\K(t,u)$.
\label{thm:bmj}
\end{thm}

Here, $u$ is the catalytic variable of the equation, which is the role played by $x$ in~\eqref{eq:bouttier-carrance-eq-notation}, and the main variable ($t$ in the theorem) would be $w$. The field $\K$ would be $\QQ(t,t_i,\wt{t}_j)$. The first obstruction in applying this theorem to~\eqref{eq:bouttier-carrance-eq-notation} is that it contains not only polynomial (or rational) functions in $x$, but also the function $Y(x)$, which is a formal power series. However, we can make use of the rational parametrization~\eqref{eq:xy-param}, to switch from $x$ to $\bar{z}=z^{-1}$ as a catalytic variable. Then~\eqref{eq:bouttier-carrance-eq-notation} writes, for $\mathcal{M}(\bar{z},w)=\wh M(x(z),w)$:
\begin{align}
\mathcal{M}(\bar{z},w)=&\mathcal{M}(\bar{z},w)(A(w)+wx(z)y(z))+wx(z)(V'(x(z))+y(z))y(z) \nonumber\\
&-w\sum_{2\leq i \leq d}t_i\Big[x(z)^i(\mathcal{M}(\bar{z},w)+V'(x(z))+Y(x(z))) \Big]_{x^{\geq 0}}.
\label{eq:tutte-m-z}
\end{align}
Note that $\mathcal{M}_0(\bar{z})=\mathcal{M}(\bar{z},0)=0$ is clearly explicit, so that~\eqref{eq:tutte-m-z} can be written as
\begin{align*}
\mathcal{M}(\bar{z},w)=&\mathcal{M}_0(\bar{z}) + w\Big(\frac{A(w)}{w}+x(z)y(z))\mathcal{M}(\bar{z},w)+(x(z)y(z))(V'(x)+y(z)) \\
& + \sum_{2\leq i \leq d}t_i\Big[x(z)^i(\mathcal{M}(\bar{z},w)+V'(x(z))+Y(x(z))) \Big]_{x^{\geq 0}} \Big).
\end{align*}
Everything in factor of $w$ in the right-hand side is an explicit rational function in $\bar{z}$ and $\mathcal{M}(\bar{z},w)$, except $A(w)/w=[\bar{x}](M(x,w))+t$, and 
\[
\sum_{2\leq i \leq d}t_i\Big[x(z)^i\mathcal{M}(\bar{z},w)\Big]_{x^{\geq 0}}
\]
which are explicit combinations of the coefficients of $\bar{x}, \bar{x}^2, \dots, \bar{x}^{d}$, in the power series $\wh M(x,w)=:M(\bar{x}) \, \in \Q[c,t,t_i,\wt{t}_j][[\bar{x},w]]$. They are first easily expressed in terms of the divided differences in $\bar{x}$. Indeed, we have
\[
\Delta^{(i)}_{\bar{x}}M=\frac{M(\bar{x})-M(0)-\bar{x}M'(0)- \dots -\bar{x}^{i-1}/(i-1)! M^{i-1}(0) }{\bar{x}^i},
\]
so that
\[
[\bar{x}^i]M=\frac{1}{i!}M^{(i)}(0)=\Delta^{(i)}_{\bar{x}}M-\bar{x}\Delta^{(i+1)}_{\bar{x}}M.
\]
Now, we claim that the divided differences in $\bar{x}$ can be written in terms of the ones in $\bar{z}$, in the following way, for any $i\geq1$:
\begin{equation}
\Delta_{\bar{x}}^{(i)}M(\bar{x}(\bar{z}))=\Phi(\bar{z})\left(\sum_{1\leq j \leq i}p_{i,j}(\bar{z})\Delta^{(j)}_{\bar{z}}\mathcal{M} \right),
\label{eq:divided-diff-x-to-z}
\end{equation}
where $p_{i,j}(\bar{z})$ are explicit polynomials in $\Q[c,\gamma,\alpha_k,\beta_l][\bar{z}]$. Indeed, as $\bar{x}(\bar{z})=\bar{z}/\Phi(\bar{z})$, where $\Phi(\bar{z})=\gamma+\sum_{0\leq k \leq \wt d -1}\alpha_k\bar{z}^k$, we have
\[
\Delta_{\bar{x}}M(\bar{x}(\bar{z}))=\frac{M(\bar{x}(\bar{z}))-M(0)}{\bar{x}(\bar{z})}=\Phi(\bar{z})\cdot\frac{\mathcal{M}(\bar{z})-\mathcal{M}(\bar{z}=0)}{\bar{z}}=\Phi(\bar{z})\Delta_{\bar{z}}\mathcal{M},
\]
since $\bar{x}(\bar{z}=0)=0$. The statement for general $i$ follows by a straightforward induction, applying the Leibniz rule for divided differences:
\[
\Delta_x(f(x)g(x))=f(0)\cdot\Delta_xg(x) + \Delta_xf(x)\cdot g(x).
\]

Equation~\eqref{eq:divided-diff-x-to-z} implies that the coefficient $[\bar{x}^i]M$ is itself an explicit polynomial in the variables $\bar{z}, \mathcal{M}, \Delta_{\bar{z}}\mathcal{M}, \dots \Delta^{(i+1)}_{\bar{z}}\mathcal{M}$ with coefficients in $\Q[c,\gamma,\alpha_k,\beta_\ell]$, so that (up to multiplication by a power of $\bar{z}$),~\eqref{eq:tutte-m-z} can indeed by written in the form of~\eqref{eq:bm-j-thm}, and Theorem~\ref{thm:bmj} applies.\\

Let us end this section with a sketch of the strategy of the \emph{kernel method}, applied to general potentials $V, \wt V$, following the exposition given in~\cite[Section 3.1]{BoJe06}. We compare it with our ``$Q=E$'' method in Section~\ref{sec: algebraicity-2-strats}.

The kernel method consists in looking for formal power series $\bar{Z}=\bar{Z}(\omega) \in \QQ[c,\gamma,\alpha_k,\beta_\ell][\omega]$ that satisfy 
\begin{equation}
\mathcal{K}(\bar{Z},\omega)=0,
\label{eq:kernel-z-bar}
\end{equation}
where $\mathcal{K}(\bar{z},\omega)=K(x(z),\omega)$. Indeed, for such a solution, equation~\eqref{eq:tutte-m-z} implies
\[
0=\mathcal{K}(\bar{Z}_i,w)\mathcal{M}(\bar{Z}_i,w)=\mathcal{R}(\bar{Z}_i, M_1, \dots, M_d, w),
\]
with $M_i=[\bar{x}^{i}]M(\bar{x})$ and $\mathcal{R}(\bar{z}, M_1, \dots, M_d, w)=\wh{R}(x(z),w)$, where the dependence on the $M_i$ is now emphasized. If we have $d$ such solutions $\bar{Z}_1, \dots, \bar{Z}_d$, counted with multiplicity, then this gives a system of $d$ equations on the $d$ unknowns $M_1, \dots, M_d$ (if a solution $\bar{Z}_i$ has non-trivial multiplicity, one needs to write the equations associated to the derivatives of $\mathcal{R}$ at $\bar{Z}_i$). We can then solve this system and insert the solutions in~\eqref{eq:tutte-m-z}, yielding an explicit polynomial equation on $\mathcal{M}(\bar{z},w)$. Theorem~\ref{thm:bmj} ensures that there is a unique solution to the polynomial system on $M_1,\dots,M_d$, as these then determine $\mathcal{M}(\bar{z},w)$.

It remains to check that we have $d$ solutions to~\eqref{eq:kernel-z-bar}. Since we have:
\begin{align*}
x(z)y(z)&=\Bigg(\gamma z +\sum_{0 \leq k \leq \wt d -1}\alpha_kz^{-k}\Bigg)\Bigg(\gamma z^{-1} +\sum_{0 \leq \ell \leq d -1}\beta_\ell z^{\ell}\Bigg)\\
&=\Bigg(\gamma +\sum_{0 \leq k \leq \wt d -1}\alpha_kz^{-k-1}\Bigg)\Bigg(\gamma +\sum_{0 \leq \ell \leq d -1}\beta_\ell z^{\ell+1}\Bigg),
\end{align*}
to get to a polynomial equation on $\bar{z}$ from $\mathcal{K}(\bar{z},w)$, we have to multiply it by $\bar{z}^d$, yielding the following power series in $w$ with polynomial coefficients in $\bar{z}$:
\[
\Psi(w,\bar{z})=\bar{z}^d(1-A(w))+w\Bigg(\gamma +\sum_{0 \leq k \leq \wt d -1}\alpha_k\bar{z}^{k+1}\Bigg)\Bigg(\gamma \bar{z}^d +\sum_{0 \leq \ell \leq d -1}\beta_\ell \bar{z}^{d-\ell-1}\Bigg).
\]
Since the constant coefficient in $w$ is $\Psi(0,\bar{z})= \bar{z}^d$,~\cite[Theorem 2]{BoJe06} implies that there are exactly $d$ solutions to $K(x(\bar{Z}),w)=0$ that are fractional power series in $w$, counted with multiplicity.

\newpage
\printbibliography

\end{document}